\documentclass[reqno,a4paper, 12pt]{amsart}
    
\usepackage{amsmath,amsfonts,amssymb,amsthm}                
\usepackage{graphicx}                 
\usepackage{dsfont}               
\usepackage{tikz}

\numberwithin{equation}{section}

\theoremstyle{definition}
\newtheorem*{definition*}{Definition}
\newtheorem*{remark}{Remark}

\theoremstyle{plain}
\newtheorem{theorem-main}{Theorem}
\newtheorem{corollary-main}{Corollary}
\newtheorem{theorem}{Theorem}[section]
\newtheorem{lemma}[theorem]{Lemma}
\newtheorem{lem}[theorem]{Lemma}
\newtheorem{proposition}[theorem]{Proposition}

\newcommand{\thmref}[1]{Theorem~\ref{#1}}

\newcommand{\lemref}[1]{Lemma~\ref{#1}}
\newcommand{\corref}[1]{Corollary~\ref{#1}}

\newcommand{\va}{\ensuremath{\alpha}}
\newcommand{\vb}{\ensuremath{\beta}}
\newcommand{\bz}{\mathbb{Z}}
\newcommand{\bt}{\mathbb{T}}
\newcommand{\br}{\mathbb{R}}

\newcommand{\toinf}{\ensuremath{\rightarrow \infty}}
\newcommand{\norm}[1]{\left\lVert#1\right\rVert}
\newcommand{\seq}[1]{\left\{#1\right\}}
\newcommand{\ip}[2]{\langle #1, #2 \rangle}
\newcommand{\1}{\mathds{1}}

\newcommand{\sect}{\textsection}
\newcommand{\piped}{parallelepiped}
\renewcommand{\leq}{\leqslant}
\renewcommand{\geq}{\geqslant}
\renewcommand{\le}{\leqslant}

\DeclareMathOperator{\mes}{mes}

\DeclareMathOperator{\vol}{Vol}
\DeclareMathOperator{\spann}{span}
\DeclareMathOperator{\adj}{adj}

\newenvironment{enumerate-math}
{\begin{enumerate}
\addtolength{\itemsep}{5pt}
}
{\end{enumerate}}

\begin{document}
\title[Sets of bounded discrepancy]{Sets of bounded discrepancy for multi-dimensional irrational rotation}

\author{Sigrid Grepstad}
\address{Department of Mathematical Sciences, Norwegian University of Science and Technology (NTNU), NO-7491 Trondheim, Norway.}
\email{\texttt{sigrid.grepstad@math.ntnu.no}}
\author{Nir Lev}
\address{Department of Mathematics, Bar-Ilan University, Ramat-Gan 52900, Israel.}
\email{\texttt{levnir@math.biu.ac.il}}

\begin{abstract}
We study bounded remainder sets with respect to an irrational rotation of the $d$-dimensional torus.
The subject goes back to Hecke, Ostrowski and Kesten who characterized the intervals with
bounded remainder in dimension one.

First we extend to several dimensions the Hecke-Ostrowski result
by constructing a class of $d$-dimensional {\piped}s of bounded remainder.
Then we characterize  the Riemann measurable bounded remainder sets
in terms of ``equidecomposability'' to such a {\piped}.
By constructing invariants with respect to this equidecomposition, we derive explicit conditions for a polytope to be a
bounded remainder set.  In particular this yields a characterization of the convex bounded remainder polygons in two dimensions.
 The approach is used to obtain several other results as well.
\end{abstract}

\date{October 21, 2014}
\thanks{Research partially supported by the Israel Science Foundation grant No. 225/13}
\maketitle


\section{Introduction}\label{sec:intro}

\subsection{}\label{sec:intro1}
Let $\va = (\alpha_1, \ldots , \alpha_d)$ be a vector in $\br^d$, such that the numbers
 $1, \alpha_1, \alpha_2, \ldots, \alpha_d$ are linearly independent over the rationals. 
It is a classical fact that under this condition, the sequence $\seq{n\va}$ is equidistributed on 
the $d$-dimensional torus $\bt^d = \br^d / \bz^d$, meaning that 
\begin{equation}\label{equid_def}
\frac{1}{n} \sum_{k=0}^{n-1} \chi_S(x+k \va) \to \mes S \quad (n \to \infty)
\end{equation}
for any $x \in \bt^d$ and any Riemann measurable set $S \subset \bt^d$ (a set $S$ is Riemann measurable if its boundary has measure zero). Here, $\chi_S$ denotes the indicator function of $S$. 

The equidistribution of the sequence $\seq{n\va}$ with respect to the set $S$ 
can be measured quantitatively by means of the \emph{discrepancy function}, defined by
\begin{equation}
\label{eq:discrepancy}
D_n(S,x) = \sum_{k=0}^{n-1} \chi_S(x+k\va) - n \mes S .
\end{equation}
Thus \eqref{equid_def} can be reformulated by saying that $D_n(S,x) = o(n)$, $n \rightarrow \infty$.

It was discovered that for certain special sets $S$, a far better estimate for the discrepancy can be given. Hecke \cite{hecke}
and Ostrowski \cite{ostrowski2, ostrowski} showed that if $I \subset \bt$ is an interval with length in $\bz \va + \bz$ then
$D_n(I,x)$ remains bounded as $n \to \infty$. The converse to this result, conjectured by Erd\H{o}s and Sz\"{u}sz in \cite{erdos}, 
was confirmed by Kesten \cite{kesten}, who proved that the discrepancy $D_n(I,x)$ is bounded only if the length of $I$ is in $\bz \va + \bz$. 

The study of this phenomenon for sets more general than intervals, and in higher dimensions, has led to the notion of a \emph{bounded remainder set} (BRS). A measurable set $S$ is called a BRS
if there is a constant $C=C(S, \va)$, such that $|D_n(S,x)| \leq C$ for every $n$ and almost every $x$.
One can show that if $\sup_n |D_n(S,x)|<\infty$ for every $x$ in some set of positive measure,
or even just for one $x$ if $S$ is Riemann measurable, then $D_n(S,x)$ is actually bounded almost everywhere
with a uniform constant, which leads to the above definition of a bounded remainder set
(see Propositions \ref{prop:rmbrs1} and \ref{prop:rmbrs2} in \sect\ref{sec:basicprop}). 

Bounded remainder sets have been studied by Sz\"{u}sz \cite{szusz, szusz2},
Furstenberg, Keynes and Shapiro \cite{furstenberg}, Petersen \cite{petersen}, Hal\'{a}sz \cite{halasz}, Oren \cite{oren}, 
Rauzy \cite{rauzy1,rauzy3}, Liardet \cite{liardet}, Ferenczi \cite{ferenczi} and others,
see also \cite{haynes,shutov,zhuravlev}.
Some authors have also considered bounded remainder sets in $\br^d$,
identifying $S$ with its image under the canonical projection $\br^d \to \bt^d$.
In this context, the discussion was restricted to sets $S$ which are
\emph{simple}, meaning that the canonical projection restricted to $S$ is injective. 
Here we will extend the discussion to bounded sets $S$ in $\br^d$
which are not necessarily simple, by considering the projection of $S$
as a \emph{multiset} on $\bt^d$ with multiplicity function
\begin{equation}
\label{eq:chis}
 \chi_S(x) = \sum_{k \in \bz^d} \1_S(x+k),
\end{equation}
where $\1_S$ is the indicator function of $S$ in $\br^d$. The definition of a BRS can be easily extended to this setting by understanding $\chi_S$ in \eqref{eq:discrepancy} to be the multiplicity function defined by \eqref{eq:chis}.

\subsection{}\label{sec:intro1.2}
By the Hecke-Ostrowski-Kesten result mentioned above, an interval $I\subset \br$ 
is a BRS if and only if its length belongs to $\bz \va + \bz$. Hartman \cite{hartman} raised the question of whether this result admits a higher-dimensional analog. 

Non-trivial examples of bounded remainder sets (parallelograms) in two dimensions
were first given by Sz\"{u}sz \cite{szusz},
which include, for instance, the parallelogram spanned by the vectors
$(\alpha_1, \alpha_2)$ and $(\alpha_1 / \alpha_2, 0)$. 
An extension of this result for ``cylindric'' sets in higher 
dimensions was later obtained by Liardet \cite{liardet} (see \sect \ref{sec:basicprop} below). 

In the first of our main results in this paper we exhibit a new class of bounded remainder {\piped}s. 
\begin{theorem-main}
Any {\piped} in $\br^d$ spanned by vectors $v_1, \ldots , v_d$ belonging
to $\bz \va + \bz^d$ is a bounded remainder set.
\label{thm:main1}
\end{theorem-main}
By a {\piped} spanned by vectors $v_1, \ldots , v_d$ we mean a set of the form 
\begin{equation*}
P = \left\{  \sum_{k=1}^d t_k v_k \, : \, 0 \leq t_k < 1 \right\},  
\end{equation*}
where $v_1, \ldots , v_d$ are linearly independent vectors in $\br^d$. 

Theorem \ref{thm:main1} may be viewed as an extension to higher dimensions
of the Hecke-Ostrowski result on bounded remainder intervals in dimension one.
See \sect\ref{sec:main1}, where we also formulate a more general version of this result (\thmref{thm:genpiped}).

Since the union of two disjoint BRS is a BRS, it follows from Theorem \ref{thm:main1} that any polytope which can be tiled by a finite number of {\piped}s spanned by vectors in $\bz \va + \bz^d$ is a BRS. In particular, this implies:
\begin{corollary-main}
\label{cor:zonotope}
Any convex, centrally symmetric polygon in $\br^2$ with vertices belonging to $\bz \va + \bz^2$ is a bounded remainder set. More generally, any zonotope in $\br^d$ with vertices belonging to $\bz \va + \bz^d$ is a bounded remainder set.
\end{corollary-main}
Recall that a zonotope is a convex polytope which can be represented as the Minkowski sum of several line segments.
Equivalently, a zonotope is a convex, centrally symmetric polytope with centrally symmetric $k$-dimensional faces for $2 \leq k \leq d-1$. Thus in $\br^2$ the zonotopes are the convex, centrally symmetric polygons.

It is known that the measure $\gamma$ of any bounded remainder set $S$ must be of the form 
\begin{equation}
\gamma = n_0 + n_1 \alpha_1 + \cdots + n_d \alpha_d , 
\label{eq:mesmustbe}
\end{equation}
where $n_0, \ldots , n_d$ are integers. This is a generalization of Kesten's theorem, see Proposition
\ref{prop:mesgiven} in \sect\ref{sec:basicprop}. Conversely, any positive number $\gamma$ of the form \eqref{eq:mesmustbe} may be realized as the measure of some bounded remainder set \cite{halasz}. It is a useful fact, as we will see below, that such a set may be chosen to be a {\piped} of the form in Theorem \ref{thm:main1}.
\begin{corollary-main}
\label{cor:main1anymes}
Let $\gamma$ be a positive number of the form \eqref{eq:mesmustbe}. Then there exists a bounded remainder {\piped} $P$,
spanned by vectors belonging to $\bz \va + \bz^d$, with $\mes P = \gamma$.
\end{corollary-main}

We also show that if, moreover, $\gamma \leq 1$, then $P$ may be chosen to be a simple set.

\subsection{} \label{sec:intro:main2}
Two measurable sets $S$ and $S'$ in $\br^d$ are said to be \emph{equidecomposable}, or \emph{scissors congruent}, if the set $S$ can be partitioned 
into finitely many measurable subsets that can be reassembled by rigid motions to form, up to measure zero, a partition of $S'$.
If $S$ and $S'$ belong to a restricted class of sets, e.g.\ they are Riemann measurable sets, or they are polytopes, then we shall
require the pieces of the partition to belong to the same class.

Equidecomposability of polytopes has received much attention, as it goes back to Hilbert's third problem -- the question of whether two 
polyhedra of equal volume are necessarily equidecomposable (by polyhedral pieces). 
In two dimensions any two polygons of equal area are equidecomposable,
but in three dimensions it was shown by Dehn that such a result is no longer true
(see \cite{boltianski} for a detailed exposition of the subject).

It is also interesting to consider a restricted notion of equidecomposability, where the pieces of the partition 
are allowed to be reassembled only by motions belonging to some given subgroup $G$ of all rigid motions.
In this context, the most well studied case is the equidecomposability with respect to the group 
of all translations of $\br^d$.

It is not difficult to show that if two sets $S$ and $S'$ are equidecomposable using only translations by vectors in $\bz \va + \bz^d$,
and if $S$ is a bounded remainder set, then so is $S'$ (see Proposition \ref{prop:shiftsubset} in \sect\ref{sec:main2}).
Our second main result establishes a converse statement in the case when the two sets are
Riemann measurable.

\begin{theorem-main}
\label{thm:main2}
Let $S$ and $S'$ be two Riemann measurable bounded remainder sets of the same measure.
Then $S$ and $S'$ are equidecomposable (by Riemann measurable pieces)
using translations by vectors belonging to $\bz \va + \bz^d$ only.
\end{theorem-main} 

In the case when $S,S'$ are polytopes, our proof yields an equidecomposition 
using pieces which are polytopes as well.

A combination of Theorems \ref{thm:main1}, \ref{thm:main2} and Corollary \ref{cor:main1anymes}
allows us to obtain a characterization of the Riemann measurable bounded remainder sets:
\begin{corollary-main}
\label{cor:equidecomp}
A Riemann measurable set $S$ in $\br^d$ is a bounded remainder set if and only if it is equidecomposable 
to some {\piped} spanned by vectors in $\bz \va + \bz^d$, using translations by vectors belonging to $\bz \va + \bz^d$.
\end{corollary-main}

This characterization, in turn, enables us to prove several other results, which will be described next.
See also \cite{ferenczi} where a different characterization of bounded remainder sets is proposed.

\subsection{}
A key idea in the study of equidecomposability of polytopes with respect to a group of motions $G$, is the concept
of \emph{additive $G$-invariants}. A function $\varphi$ defined on the set of all polytopes in $\mathbb R^d$ is said to be an 
additive $G$-invariant if (i) it is additive, namely if $S_1, S_2$ are two polytopes with disjoint interiors then
$\varphi(S_1 \cup S_2) = \varphi(S_1) + \varphi(S_2)$; and (ii) it is invariant under motions of the
group $G$, that is $\varphi(S) = \varphi(g(S))$ whenever $S$ is a polytope and $g \in G$.
It is clear that a necessary condition for two polytopes $S$ and $S'$ to be $G$-equidecomposable 
is that $\varphi(S) = \varphi(S')$ for any additive $G$-invariant $\varphi$. A general problem is to construct a
``complete set'' of additive $G$-invariants, namely invariants which together provide a necessary and
sufficient condition for two polytopes of the same volume to be $G$-equidecomposable.

In his solution to Hilbert's third problem, Dehn found additive invariants with respect to the group of all 
rigid motions, which allowed him to show that a cube and a regular tetrahedron in $\br^3$
are not equidecomposable. Dehn invariants for polytopes in $\br^d$ have also been studied, and
shown to form a complete set in dimension $d=3,4$ \cite{jessen2, sydler}, but it remains
an open problem to explicitly describe invariants which give a necessary and
sufficient condition for equidecomposability in dimension $d \geq 5$.

Additive invariants with respect to the group of all translations of $\br^d$ were introduced by Hadwiger 
and proved to form a complete set in any dimension \cite{hadwiger, glur, jessen, sah}.

By constructing Hadwiger-type invariants with respect to
the group of translations by vectors in $\bz \va + \bz^d$ (a countable,
dense subgroup of all the translations), we can derive from
Corollary \ref{cor:equidecomp} explicit conditions for a polytope $S$ to
be a bounded remainder set. This applies already in dimension one, where it yields the
characterization due to Oren \cite{oren} of the finite unions of intervals with bounded remainder:

\emph{Let $S \subset \br$ be the union of $N$ disjoint intervals $[ a_j , b_j ]$, $1 \leq j \leq N$. Then $S$ is a bounded remainder set if and only if there exists a permutation $\sigma$ of $\{1, \ldots , N\}$ such that $b_{\sigma (j)} - a_j \in \bz \va + \bz$ for each $1 \leq j \leq N$.}

 In two dimensions we obtain the following characterization of the convex polygons with bounded remainder:

\begin{theorem-main}
\label{thm:cvxpg}
Let $S$ be a convex polygon in $\br^2$. Then $S$ is a bounded remainder set if and only if
it is centrally symmetric, and for each pair of parallel edges $e$ and $e'$ 
the following two conditions are satisfied:
\begin{enumerate-math}
\item
\label{item:facets1}
There are two points, one on the edge $e$ and the other on the edge $e'$,
which differ by a vector in $\bz \va + \bz^2$;
\item
\label{item:facets2}
If the midpoints of the edges $e$ and $e'$ do not differ by a vector in $\bz \va + \bz^2$,
then the vectors $e, e'$ themselves belong to $\bz \va + \bz^2$.
\end{enumerate-math}
\end{theorem-main}

Notice that conditions \ref{item:facets1} and \ref{item:facets2} are satisfied if the vertices of $S$ lie in $\bz \va + \bz^2$.
See also \cite{bolle}, where a similar characterization is given of the convex polygons $S$
for which the function $\chi_S$ is constant a.e.\ (in particular, such a polygon is a BRS).

In higher dimensions we also get non-trivial conditions for a polytope $S$ to be a bounded remainder set. For example:

\begin{theorem-main}
\label{thm:cvxptpsym}
For a convex polytope $S$ in $\br^d$ to be a bounded remainder set,
it is necessary that $S$ is centrally symmetric and has centrally symmetric $(d-1)$-dimensional faces.
\end{theorem-main}

Recall that we also gave a sufficient condition for a convex polytope to be a BRS,
namely that it is a zonotope with vertices belonging to $\bz \va + \bz^d$ (\corref{cor:zonotope}).
In three dimensions, we thus obtain the following

\begin{corollary-main}
\label{cor:zonohedron}
Let $S$ be a convex polyhedron in $\br^3$ with vertices belonging to $\bz \va + \bz^3$.
Then $S$ is a bounded remainder set if and only if it is a zonohedron, namely
$S$ is centrally symmetric and has centrally symmetric faces.
\end{corollary-main}

As another application to our approach, we show how to derive a result due to
Liardet \cite{liardet} that characterizes the bounded remainder
multi-dimensional rectangles with sides parallel to the coordinate axes:

\emph{If $S \subset \br^d$ is the product of $d$ intervals $I_1 \times \cdots \times I_d$
then $S$ is a bounded remainder set if and only if
the length of one of the intervals $I_j$ belongs to $\bz \va_j + \bz$,
while the lengths of all the other intervals belong to $\bz$.}

The above results merely illustrate how the characterization of bounded remainder polytopes 
in terms of equidecomposability may be used in some special cases of interest.
The general conditions, formulated in terms of Hadwiger-type invariants with respect to
the group of translations by vectors in $\bz \va + \bz^d$, are presented in \sect \ref{sec:hadwiger}.

\subsection{}
We continue with some further applications to the characterization of Riemann measurable bounded remainder sets.

Let us indicate the property that the numbers $1, \alpha_1, \alpha_2, \ldots, \alpha_d$ 
are linearly independent over the rationals by saying that $\va$ is an \emph{irrational vector}. 
Given two irrational vectors $\va$ and $\vb$ in $\br^d$, one may attempt to find relations between the bounded remainder sets corresponding to the vector $\va$ and those corresponding to the vector $\vb$.

For example, it is known that the condition $\va \in \bz \vb + \bz^d$ implies that
any BRS with respect to $\va$ is also a BRS with respect to $\vb$
(see Proposition \ref{prop:bversusa} in \sect\ref{sec:basicprop}).

We consider a more general situation of an invertible linear map $T$ 
on $\br^d$, which maps bounded remainder 
sets with respect to $\va$ to bounded remainder sets with respect to $\vb$.
The following result describes completely the linear maps satisfying
this condition with respect to Riemann measurable bounded remainder sets.

\begin{theorem-main}
\label{thm:main3}
Let $\va$ and $\vb$ be two irrational vectors in $\br^d$, and let $T$ be an invertible linear map on $\br^d$. 
In order for the image under $T$ of every Riemann measurable BRS with respect to $\va$
to be a BRS with respect to $\vb$, it is necessary and sufficient that
\begin{equation}
\label{eq:tiff}
T(\bz \va + \bz^d) \subset \bz \vb + \bz^d.
\end{equation}
\end{theorem-main}

In particular, using this for the identity map allows us to characterize the situation when 
every BRS with respect to $\va$ is also a BRS with respect to $\vb$:

\begin{corollary-main}
\label{cor:brsab}
Let $\va$ and $\vb$ be two irrational vectors in $\br^d$. In order for 
every bounded remainder set with respect to $\va$ to be 
a bounded remainder set also with respect to $\vb$, it is necessary and sufficient that
$\va \in \bz \vb + \bz^d$.
\end{corollary-main}

See \sect\ref{sec:linear}, where we also give a parametrization of all the couples $(\vb,T)$ satisfying condition
\eqref{eq:tiff} for a given irrational vector $\va$ (\thmref{thm7.1}), and prove a version of \thmref{thm:main3} for not
necessarily Riemann measurable bounded remainder sets (\thmref{thm7.2}).

\subsection{}
Another result is concerned with the notion of the \emph{transfer function} of a bounded remainder set.
It is well-known that a measurable set $S$ is a BRS if and only if there exists a bounded, measurable function $g$ on the torus $\bt^d$ such that
\begin{equation}
\chi_S(x) - \mes S = g(x) - g(x-\va) \quad \text{a.e.}
\label{eq:cobound}
\end{equation} 
This equation is known as the \emph{cohomological equation} for $S$ (sometimes this equation is considered with $g(x+\alpha)$ instead of $g(x-\alpha)$ on the right hand side, but it is easy to change from one convention to the other). The function $g$ is unique a.e.\ up to an additive constant, and is called the \emph{transfer function} for $S$.

Our proof of Theorem \ref{thm:main1} is based on the explicit construction of a bounded transfer function
for a {\piped} spanned by vectors in $\bz \va + \bz^d$. Theorem \ref{thm:main2}, in turn, is proved by
the explicit construction of an equidecomposition of any two Riemann measurable bounded remainder sets.
These results thus essentially enable us to determine the transfer function $g$ of a Riemann measurable BRS explicitly. 
In particular, we obtain:

\begin{theorem-main}
If $S \subset \br^d$ is a Riemann measurable bounded remainder set, then it has a Riemann integrable transfer function.
\label{thm:Rtransfer}
\end{theorem-main}

\subsection{}
The rest of the paper is organized as follows. In \sect \ref{sec:basicprop} we survey the basic results on bounded remainder sets.
In \sect \ref{sec:main1} we study {\piped}s of bounded remainder, and in particular 
prove that any {\piped} spanned by vectors in $\bz \va + \bz^d$ is a bounded remainder set.
In \sect \ref{sec:main2} we characterize the Riemann measurable bounded remainder 
sets in terms of equidecomposability with respect to translations by vectors in $\bz \va + \bz^d$.
In \sect \ref{sec:hadwiger} we study bounded remainder polytopes using Hadwiger-type invariants.
In \sect \ref{sec:linear} we describe linear maps on $\br^d$ which map bounded remainder 
sets with respect to $\va$ to bounded remainder sets with respect to another irrational vector $\vb$.
In the last \sect\ref{sec:remarks} we give additional remarks and mention some open problems.

\subsection*{Acknowledgement}
We are grateful to Lev Buhovski, Leonid Polterovich and Barak Weiss for their help in various respects related to this work.


\section{Preliminaries}\label{sec:basicprop}

\subsection{Notation and terminology}\label{sec:terminology}
A vector $\va = (\alpha_1, \alpha_2, \ldots , \alpha_d) \in \br^d$ will be called an \emph{irrational vector} if the numbers $1, \alpha_1, \ldots , \alpha_d$ are linearly independent over the rationals. Thus $\va$ is irrational if and only if the points $\{ n \va\}$ are dense on the torus $\bt^d = \br^d / \bz^d$. 

For a bounded, measurable set $S \subset \br^d$ we denote by 
\begin{equation*}
\chi_S(x) = \sum_{k \in \bz^d} \1_S (x+k)
\end{equation*}
the multiplicity function of the projection of $S$ on $\bt^d$. Since $\chi_S$ is $\bz^d$-periodic, we shall consider it as a function on $\bt^d$. We say that $S$ is a \emph{simple set} if $\chi_S$ is $\{0,1\}$-valued. 

We say that $S$ is a \emph{bounded remainder set} (BRS) if there is a constant $C= C(S, \va)$ such that
\begin{equation}
\label{eq:termdisc}
\Big| \sum_{k=0}^{n-1} \chi_S(x + k\va) - n \mes S \Big| \leq C
 \quad (n=1,2,3,\dots) 
 \quad \text{a.e.\ $x \in \bt^d$.} 
\end{equation}

A measurable function $g$ on $\bt^d$ will be called a \emph{transfer function} for $S$ if 
\begin{equation}
\label{eq:termcobound}
\chi_S(x) - \mes S = g(x) - g(x-\alpha) \quad \text{a.e.}
\end{equation}

We denote by $\widehat{f} (n)$, $n \in \bz^d$, the Fourier coefficients of a function $f$ on $\bt^d$. It is easy to check that for a bounded set $S \subset \br^d$ we have
\begin{equation}
\label{eq:chift}
\widehat{\chi}_S (n) = \int_S e^{-2\pi i \ip{n}{x}} dx,
\end{equation}
and that the cohomological equation \eqref{eq:termcobound} is equivalent to the condition
\begin{equation}
\label{eq:termcoboundft}
\widehat{g}(n) = \frac{\widehat{\chi}_S (n)}{1 - e^{-2 \pi i \ip{n}{\alpha}}}, \quad n \in \bz^d \setminus \{0\}.
\end{equation}

A bounded set $S \subset \br^d$ is called \emph{Riemann measurable} if its boundary has measure zero, or equivalently, if its indicator function $\1_S$ is a Riemann integrable function. 

\subsection{Basic facts on bounded remainder sets}\label{sec:bp}
Here we survey the basic results on bounded remainder sets.
These results are basically well-known, but in order to make the exposition
complete and self-contained we have included short proofs of them. 

We start with two propositions clarifying why the definition of a BRS is indeed a natural one. The proofs essentially
follow an argument of Petersen \cite{petersen} (see also \cite{halasz}).

\begin{proposition}\label{prop:rmbrs1}
Let $S \subset \br^d$ be a bounded, measurable set, and suppose that for each $x$ in some set of positive measure we have
\begin{equation}
\sup_{n} \left| \sum_{k=0}^{n-1} \chi_S(x+ k \va) - n \mes S \right| < \infty .
\label{eq:messet}
\end{equation}
Then $S$ is a BRS. \label{item:anyset}
\end{proposition}

\begin{proposition}\label{prop:rmbrs2}
Let $S \subset \br^d$ be a bounded, Riemann measurable set, and suppose that
there exists at least one $x$ satisfying \eqref{eq:messet}. Then $S$ is a BRS. \label{item:Rset}
\end{proposition}

\begin{proof}[Proof of Proposition \ref{prop:rmbrs1}]
Let $f(x) = \chi_S(x) - \mes S$, and denote 
\begin{equation}
S_n(x) = \sum_{k=0}^{n-1} f(x+k\va) .
\label{eq:Sn}
\end{equation}
It is given that $\sup_n |S_n(x)| < \infty$ for all $x$ in some set $E$ of positive measure in $\bt^d$. Then there is a
constant $M$ such that $\sup_n |S_n(x)| \leq M$ in some subset $E'\subset E$ of positive measure. Since
\begin{equation}
|S_n(x+j\va)| = |S_{n+j}(x)-S_j(x)| \leq 2M
\label{eq:2M}
\end{equation}
for $x \in E'$, it follows that $|S_n| \leq 2M$ on the union $\bigcup_{j=1}^{\infty} (E'+j\va)$, 
a set of full measure in $\bt^d$. Hence $S$ is a BRS.
\end{proof}

\begin{proof}[Proof of Proposition \ref{prop:rmbrs2}]
With the same notations, we are given that there is a point $x_0$ 
such that $M := \sup_n |S_n(x_0)| < \infty$. By \eqref{eq:2M} we obtain $|S_n| \leq 2M$ on the set $\{ x_0 + j\va \}$, a dense subset of $\bt^d$. Since $S$ is Riemann measurable, $S_n$ is continuous at almost every point. It follows that $|S_n| \leq 2M$ a.e., hence $S$ is a BRS. 
\end{proof}
 
We proceed with the equivalence of the bounded remainder property and the existence of a bounded transfer function. 

\begin{proposition}
\label{prop:trnsequiv}
For a bounded, measurable set $S \subset \br^d$, the following are equivalent:
\begin{enumerate-math}
\item \label{item:brs} $S$ is a bounded remainder set. 
\item \label{item:transfer} There exists a (real-valued) bounded, measurable function $g$ on $\bt^d$ satisfying the cohomological equation \eqref{eq:termcobound}. 
\end{enumerate-math}
\end{proposition}

\begin{proof}
Let $f(x) = \chi_S(x) - \mes S$.
The cohomological equation \eqref{eq:termcobound} is equivalent to 
\begin{equation}
\label{eq:termcoboundclassical}
f(x) = h(x) - h(x+\va) \quad \text{a.e.,}
\end{equation}
where $g$ and $h$ are related by $g(x) = -h(x+\va)$.
Suppose first that there exists a bounded function $h$ satisfying \eqref{eq:termcoboundclassical}.
Then the sum $S_n$ defined by \eqref{eq:Sn} becomes
$S_n(x) = h(x) - h(x+ n\va)$, and so $|S_n(x)| \leq 2 \norm{h}_{\infty}$ a.e.
Hence $S$ is a BRS.

Conversely, assume that $S$ is a BRS. Then $|S_n(x)| \leq C$ a.e., so the function
\begin{equation*}
h(x) := \liminf_{n \toinf} S_n(x)
\end{equation*}
belongs to $L^{\infty}(\bt^d)$. Taking the $\liminf$ as $n \toinf$ of both sides in the equality 
\[
S_n(x+\va) = S_{n+1}(x) - f(x)
\]
yields
$h(x+\va) = h(x)-f(x)$ a.e., so we obtain the required condition \eqref{eq:termcoboundclassical}.
\end{proof}

\begin{remark}
It is easy to see that if $S$ is a BRS then the smallest constant $C = C(S,\alpha)$ satisfying \eqref{eq:termdisc} is
$C = \operatorname{ess \, sup} g - \operatorname{ess \, inf} g$, where $g$ is the
transfer function of $S$.
\end{remark}

The next result imposes an arithmetical restriction on the measure of a bounded remainder set $S$.
In the case when $S$ is an interval on $\br$, this is Kesten's theorem \cite{kesten}.
The proof given is based on an argument due to Furstenberg, Keynes and Shapiro \cite{furstenberg} and Petersen \cite{petersen} (see also \cite{halasz}).

\begin{proposition}
\label{prop:mesgiven}
Let $S$ be a BRS. Then there exist integers $n_0, n_1, \ldots , n_d$ such that 
\begin{equation}
\mes S = n_0 + n_1 \alpha_1 + \cdots + n_d \alpha_d. 
\label{eq:mesgiven}
\end{equation}
\end{proposition}
\begin{proof}
Let $g$ be the transfer function in \eqref{eq:termcobound}, and define $\tau (x) :=  \exp 2\pi i g(x)$. We have
\begin{equation*}
\tau (x-\alpha) = e^{2 \pi i g(x-\alpha)} = e^{2 \pi i (g(x)-\chi_S(x)+ \mes S)} = \tau (x) e^{2\pi i \mes S} \quad \text{a.e.,}
\end{equation*}
since $\chi_S(x)$ is integer-valued. This shows that $\tau(x)$ is an eigenfunction of the irrational rotation by $\alpha$ with eigenvalue $\exp 2 \pi i \mes S$. All eigenvalues are known to be of the form $\exp 2 \pi i \ip{n}{\va}$ where $n \in \bz^d$, so the conclusion follows.
\end{proof}

\begin{remark}
Since the eigenfunction corresponding to the eigenvalue $\exp 2 \pi i \ip{n}{\va}$ is the exponential
$\exp -2 \pi i \ip{n}{x}$, the proof also shows that the transfer function $g(x)$ (extended to $\br^d$ as a
periodic function) is equal, up to an additive constant,
to the difference of an integer-valued function and the linear function $x \mapsto \ip{n}{x}$.
\end{remark}

We conclude this section with a simple, yet useful, relation between bounded remainder sets corresponding to different irrational vectors.
\begin{proposition}
\label{prop:bversusa}
Let $\beta \in \bz \va + \bz^d$, $\beta \notin \bz^d$. If $S \subset \br^d$ is a BRS with respect to $\beta$, then it is also a BRS with respect to $\va$.
\end{proposition}
\begin{proof}
Let $\beta = q\va + p$, where $0 \neq q \in \bz$ and $p \in \bz^d$. Suppose that $q>0$ (the case $q<0$ is similar). Since $S$ is a BRS with respect to $\beta$, there is a bounded function $g$ such that
\begin{equation*}
\chi_S(x)-\mes S = g(x)-g(x-\beta) \quad \text{ a.e. } x \in \bt^d \text{ .}
\end{equation*}
From the periodicity of $g$ we have
\begin{equation*}
g(x)-g(x-\beta) = g(x) - g(x-q\alpha) = g'(x) - g'(x-\alpha) ,
\end{equation*}
where $g'(x) := \sum_{k=0}^{q-1} g(x-k\alpha)$. Hence, we see that 
$g'$ is a bounded transfer function with respect to $\alpha$ for the set $S$, so $S$
is a BRS with respect to $\alpha$.
\end{proof}

\subsection{Hecke-Ostrowski and Sz\"{u}sz-Liardet constructions of BRS}\label{sec:hsl}
Next we recall the constructions of bounded remainder sets due to Hecke \cite{hecke}, Ostrowski \cite{ostrowski2, ostrowski}, 
Sz\"{u}sz \cite{szusz} and Liardet \cite{liardet}. Simple proofs, based on an explicit construction of the transfer function, are included. 
The proofs show that these results, initially formulated for simple sets, can be extended to sets in $\br^d$ which are not necessarily simple.

The first result on bounded remainder sets is due to Hecke \cite{hecke} and Ostrowski \cite{ostrowski2, ostrowski}, who identified the bounded remainder intervals in one dimension. The proof given below is well-known.

\begin{theorem}[Hecke-Ostrowski]
\label{thm:hecke}
Any interval $I \subset \br$ with length in $\bz \alpha + \bz$ is a BRS.
\end{theorem}

\begin{proof} 
As the bounded remainder property is independent of position, we may assume that $I=[0,\beta)$, where $\beta \in \bz \alpha + \bz$, $\beta \notin \bz$ (the case $\beta \in \bz$ is trivial). By Proposition \ref{prop:bversusa} it is sufficient to show that $I$ is a BRS with respect to $\beta$. We will show that $g(x):= - \{ x\}$ is a transfer function for $I$, where $\{ x\}$ denotes the fractional part of a number $x$. Indeed, the function $g(x)-g(x-\beta)$ has a jump discontinuity of magnitude $+1$ at $x = 0$ and of magnitude $- 1$ at $x= \beta$, it is constant between these jumps, and has zero integral on $\bt$. These properties determine it uniquely as $\chi_I(x)- \mes I$, and so the claim is proved.
\end{proof}

Recall that by Kesten's theorem, for an interval to be a BRS it is not only sufficient, but also necessary, that its length belongs to $\bz \va + \bz$ (Proposition \ref{prop:mesgiven}).

In two dimensions the first non-trivial examples of bounded remainder sets were given by Sz\"{u}sz \cite{szusz}, who constructed a family of bounded remainder parallelograms. 

\begin{theorem}[Sz\"{u}sz]
\label{thm:szusz}
Let $v = (v_1, v_2) \in \bz \va + \bz^2$, $v \notin \bz^2$, and let $\sigma \in \bz v_1/v_2 + \bz$.
Then the parallelogram spanned by the vectors $(v_1, v_2)$ and $(\sigma,0)$ is a BRS.
\end{theorem}

\noindent
(to be more precise, in \cite{szusz} this was proved in the special case when $0<v_1<v_2<1$ and $\sigma=v_1/v_2$, but this case is not essentially different from the more general one above.)

This construction was extended to higher dimensions by Liardet \cite[Theorem 4]{liardet}, who showed how one can
obtain ``cylindric'' BRS in dimension $d$ from BRS in dimension $d-1$.
Let us denote a point in $\br^d = \br^{d-1} \times \br$ as $(x,y)$ where $x \in \br^{d-1}$ and $y \in \br$.

\begin{theorem}[Liardet]
\label{thm:onedimup}
Let $v = (v_1, \ldots , v_d) \in \bz \va + \bz^d$, $v \notin \bz^d$, and let $\Sigma \subset \br^{d-1}$ be a BRS with respect to the vector $(v_1/v_d, v_2/v_d, \ldots, v_{d-1}/v_d)$. Then the set
\begin{equation}
S= S(\Sigma, v) =\left\{ (x,0) + tv \, : \, x \in \Sigma \, , \, 0 \leq t < 1 \right\} , 
\label{eq:S}
\end{equation}
is a BRS (with respect to $\va$).
\end{theorem}

Starting with a Hecke-Ostrowski type interval, and applying Theorem \ref{thm:onedimup} iteratively, one can construct non-trivial bounded remainder sets
in any dimension. In two dimensions this yields the bounded remainder parallelograms given by \thmref{thm:szusz}.

The proof below differs from those in \cite{liardet,szusz} in that we do not estimate the discrepancy of $S$ directly, but rather explicitly construct the transfer function for $S$ and use Fourier coefficients to verify that it satisfies the cohomological equation.

\begin{proof}[Proof of Theorem \ref{thm:onedimup}]
We have $S(\Sigma, v) = S(\Sigma, -v) + v$, so we may restrict ourselves to the case when $v_d$
(the last entry in the vector $v$) is positive. 
By Proposition \ref{prop:bversusa} it will be enough to show that $S$ is a BRS with respect to $v$. 

Let $v_0=(v_1, v_2, \ldots, v_{d-1})$ be the vector in $\br^{d-1}$ consisting of the first $d-1$ entries of $v$. 
We wish to find a bounded function $g$ on $\bt^d$ satisfying
the cohomological equation
\begin{equation}
\chi_S(x,y) - \mes S  =  g(x,y)-g(x-v_0, y-v_d) \quad \text{for a.e. } (x,y) \in \bt^{d-1}\times\bt.
\label{eq:coeffmatch}
\end{equation}
We use \eqref{eq:termcoboundft} to reformulate this equation in terms of the Fourier coefficients
$\widehat{g}(m,n)$ and $\widehat{\chi}_S(m,n)$, where $(m,n) \in \bz^{d-1} \times \bz$. The latter
are easy to calculate using \eqref{eq:chift} and \eqref{eq:S}.
This yields that the cohomological equation \eqref{eq:coeffmatch} is equivalent to the condition
\begin{equation}
\label{eq:lig2}
\widehat{g}(m,n) = \frac{\widehat{\chi}_{\Sigma}(m) }{2\pi i \left(\ip{m}{v_0}/v_d + n\right)},
\quad (m,n) \neq (0,0).
\end{equation}

We know that $\Sigma$ is a BRS with respect to the vector $v_0/v_d$, so it admits a bounded transfer function $h: \bt^{d-1} \to \br$ satisfying
\begin{equation}
\label{eq:htransferft}
\widehat{h}(m) = \frac{\widehat{\chi}_{\Sigma}(m)}{1 - e^{-2\pi i \ip{m}{v_0}/v_d}},
\quad m \neq 0.
\end{equation}
Combining \eqref{eq:lig2} and \eqref{eq:htransferft} we thus get  that  equation \eqref{eq:coeffmatch} is equivalent to 
\begin{alignat}{2}
& \label{eq:lig2m}
\widehat{g}(m,n) = \widehat{h}(m) \cdot \frac{1 - e^{-2\pi i \ip{m}{v_0}/v_d}}{2\pi i \left(\ip{m}{v_0}/v_d + n\right)}, \quad \quad& m\neq0,\\[8pt]
& \label{eq:lig2n}
\widehat{g}(m,n) = \frac{\mes \Sigma}{2\pi i n}, \quad \quad & n \neq 0, \; m =0.
\end{alignat}

Now consider the bounded function $g$ defined by
\begin{equation*}
\label{eq:tfcylinder}
g(x,y) = h\big(x-\frac{v_0}{v_d} \{y\}\big) - \mes \Sigma \cdot \{ y \},
\end{equation*}
where $\{ y \}$ denotes the fractional part of $y$. 
Note that $g$ is $\bz^d$-periodic, so it may be considered as a function on $\bt^d$.
It is straightforward to check that the Fourier coefficients of this function $g$ 
satisfy \eqref{eq:lig2m} and \eqref{eq:lig2n}. Hence $g$ is a bounded transfer function for $S$,
so the set $S$ is a BRS with respect to $v$. 
\end{proof}

\begin{remark}
Actually, in \cite{liardet} a more general version of Theorem \ref{thm:onedimup} was formulated, where the cylindric set $S$ does not necessarily have $\Sigma \times \{ 0 \}$ as its basis. Namely, under the conditions of Theorem \ref{thm:onedimup}, the set $S$ of the form 
\begin{equation*}
S = \left\{ (x, 0) + tv \, : \, x \in \Sigma \, , \, \varphi(x) \leq t < \varphi(x) +1 \right\}
\end{equation*}
is a BRS, where $\varphi : \Sigma \rightarrow \br$ is any bounded, measurable function.
This is deduced from the  special case above  by an argument similar to Proposition \ref{prop:shiftsubset} in \sect\ref{sec:main2} (see \cite[p. 277]{liardet}).
\end{remark}


\section{Parallelepipeds of bounded remainder}\label{sec:main1}
This section is devoted mainly to the proof of Theorem \ref{thm:main1}, which states that any {\piped} $P$ spanned by vectors in $\bz \va + \bz^d$ is a bounded remainder set. In the end of the section we also prove
Corollaries \ref{cor:zonotope} and \ref{cor:main1anymes} (see \sect \ref{sec:intro})
on zonotopes of bounded remainder, and on parallelepipeds of bounded remainder with pre-given measure.
We also formulate an extension of Theorem \ref{thm:main1} which provides a
more general construction of bounded remainder {\piped}s.

The proof of Theorem \ref{thm:main1} is based on the explicit construction of a bounded transfer function $g$ satisfying the cohomological equation
\begin{equation}\label{eq:coboundP}
 \chi_P(x)- \mes P = g(x) - g(x-\va) \quad \text{a.e.}
\end{equation}
Moreover, we show that $g$ is Riemann integrable, a fact that will be needed
later on for the proof of Theorem \ref{thm:Rtransfer}. We thus reformulate the result as follows.
\begin{theorem}
\label{thm:main1specific}
Let $P$ be a {\piped} in $\br^d$ spanned by vectors $v_1, \ldots , v_d$ belonging to $\bz \va + \bz^d$. Then $P$ is a bounded remainder set,
and moreover, $P$ admits a Riemann integrable transfer function.
\end{theorem}

Below we will assume $d\geq 2$. The {\piped} $P$ is thus given by
\begin{equation*}
P=P(v_1,\ldots , v_d) = \left\{ \sum_{k=1}^d t_k v_k \, : \, 0 \leq t_k < 1 \right\},  
\end{equation*}
where $v_1, \ldots , v_d$ are linearly independent vectors of the form
\begin{equation}
\label{eq:vqr}
v_k = q_k \va + p_k, \quad q_k \in \bz, \; p_k \in \bz^d.
\end{equation}

\subsection{Fourier series of the transfer function}
Recall that by \eqref{eq:termcoboundft} the cohomological equation \eqref{eq:coboundP} is equivalent to the condition
\begin{equation*}
 \widehat{g}(\lambda) = \frac{\widehat{\chi}_P(\lambda)}{1- e^{-2\pi i \ip{\va}{\lambda}}},
\quad \lambda \in \bz^d \setminus \{ 0\}.
\end{equation*}
The Fourier coefficients $\widehat{\chi}_P(\lambda)$ are easy to calculate; using \eqref{eq:chift} we have
\begin{equation*}
\widehat{\chi}_P (\lambda) = \int_P e^{-2\pi i \ip{\lambda}{x}} dx = D \prod_{k=1}^d \widehat{\1}_I \left( \ip{v_k}{\lambda}\right),
\end{equation*}
where $D := |\det (v_1, \ldots, v_d)| = \mes P$, and 
\begin{equation*}
\widehat{\1}_I(\xi) =\frac{1-e^{-2\pi i \xi}}{2 \pi i \xi} \quad (\xi \in \br) 
\end{equation*} 
is the Fourier transform of the indicator function of the interval $I =[0,1)$. 
Our goal is thus to show that there is a Riemann integrable (and, in particular, bounded)
function $g$ with Fourier series
\begin{equation}\label{1}
\sum_{\lambda\in\mathbb Z^d}c(\lambda)e^{2\pi i\langle \lambda,x\rangle},
\end{equation}
where the coefficients $\{c(\lambda)\}$ are given by
\begin{equation}
\label{eq:fouriercoeff}
c(\lambda) =  \frac{D}{1-e^{-2\pi i \ip{\va}{\lambda}}} \cdot \prod_{k=1}^d \widehat{\1}_I \left( \ip{v_k}{\lambda}\right),
\quad \lambda \in \bz^d \setminus \{ 0 \},
\end{equation}
and where $c(0)$ may be any real number.

\subsection{Directional derivatives}
For a vector $h=(h_1,h_2,\dots, h_d)\in\mathbb R^d,$ we denote by $\partial/\partial h$ the differentiation operation with respect to $h.$ Thus
$$\frac{\partial}{\partial h}= h_1 \frac{\partial}{\partial x_1}+h_2 \frac{\partial}{\partial x_2}+\cdots +h_d \frac{\partial}{\partial x_ d}.$$
By applying this operation term-by-term to the formal Fourier series \eqref{1} we obtain another formal series
\begin{equation}\label{2}
\sum_{\lambda\in\mathbb Z^d}2\pi i\langle h,\lambda\rangle c(\lambda)e^{2\pi i\langle\lambda,x\rangle}.
\end{equation}
Let us identify the series \eqref{2} in the case when $h$ is one of the vectors $v_1,\dots,v_d$ spanning the parallelepiped $P$.
Recall that these vectors are of the form \eqref{eq:vqr}. Fix $1\le r\le d$, and assume 
for the moment that $q_r>0.$ Since $p_r\in\mathbb Z^d,$ we have
\begin{align*}
2\pi i\langle v_r,\lambda\rangle  \widehat{\1}_I(\langle v_r,\lambda\rangle) 
&=1-e^{-2\pi i\langle v_r,\lambda\rangle}=1-e^{-2\pi iq_r\langle \alpha,\lambda\rangle}\\
&=(1-e^{-2\pi i\langle \alpha,\lambda\rangle}) \sum_{j=0}^{q_r-1}e^{-2\pi ij\langle\alpha,\lambda\rangle}.
\end{align*} 
Hence, it follows from \eqref{eq:fouriercoeff} that for $\lambda\in\mathbb Z^d\setminus\{0\}$ we have
\begin{equation}\label{3}
2\pi i\langle v_r,\lambda\rangle c(\lambda)=D\cdot\prod_{k\ne r}\widehat{\1}_I(\langle v_k,\lambda\rangle)\cdot
\sum_{j=0}^{q_r-1}e^{-2\pi ij\langle\alpha,\lambda\rangle}.
\end{equation}
Let $\theta_k$ $(1\le k\le d)$ denote the image of the Lebesgue measure on the interval $I=[0,1)$ under the mapping $t\mapsto t  v_k$ into $\mathbb T^d.$ It is easy to verify that
$$\widehat \theta_k(\lambda)=\widehat{\1}_I(\langle v_k,\lambda\rangle),\quad \lambda \in\mathbb Z^d.$$
It thus follows from \eqref{3} that if $h=v_r,$ then \eqref{2} is the Fourier series of the measure
\begin{equation}\label{4}
D \cdot \Big(\prod_{k\ne r}\theta_k- dx\Big)\ast \sum_{j=0}^{q_r-1}\delta_{j\alpha}
\end{equation}
where by $\prod_{k\ne r} \theta_k$ we mean the convolution of the measures $\theta_k$ $(k \neq r)$, 
and $dx$ is the Lebesgue measure on $\mathbb T^d$
(the measure $dx$ enters into the formula \eqref{4} due to the vanishing of the central Fourier coefficient in \eqref{2}).

To describe the measure \eqref{4} more explicitly, let $\Pi_{rj}$ denote the $(d-1)$-dimensional parallelepiped spanned by the vectors $\{v_k\}_{k\ne r}$ and based at the point $j\alpha$ on $\mathbb T^d,$ that is
\begin{equation}
\label{eq:pirj}
\Pi_{rj} = \Big\{ j \alpha + \sum_{k\neq r} t_k v_k \, : \, 0 \leq t_k < 1 \Big\},  
\end{equation}
and let $\sigma_{rj}$ denote the $(d-1)$-dimensional volume measure on $\Pi_{rj}.$ 
Then the measure \eqref{4} is seen to be equal to
$$\frac{D}{D_r}\sum_{j=0}^{q_r-1}\sigma_{rj} - D  q_r\,  dx,$$
where $D_r$ denotes the common $(d-1)$-dimensional volume of the parallelepipeds $\Pi_{rj}$ corresponding to a fixed $r.$
Observe that the first term consists of the singular part of the measure, while the second term
is the absolutely continuous part.

In a similar way, one can identify \eqref{2} with $h=v_r$ as the Fourier series of a measure also in the cases when $q_r<0$ and $q_r=0.$ This yields the following

\begin{lem}\label{lem3.2}
For each $1\le r\le d,$ the series
$$\sum_{\lambda\in\mathbb Z^d} 2 \pi i\langle v_r,\lambda\rangle c(\lambda)e^{2\pi i\langle\lambda,x\rangle}$$
is the Fourier series of the measure
\begin{alignat*}{2}
&\frac{D}{D_r}\sum_{j=0}^{q_r-1}\sigma_{rj}-D  q_r \, dx, \quad \quad &q_r>0, \\
-&\frac{D}{D_r}\sum_{j=q_r}^{-1}\sigma_{rj}-D  q_r \, dx, \quad \quad &q_r<0, \\[8pt]
&0, \quad \quad &q_r=0, 
\end{alignat*}
on the torus $\mathbb T^d.$
\end{lem}

\subsection{Gradient of the transfer function}
 With the help of \lemref{lem3.2}, we can also identify the vector-valued series
 \begin{equation}\label{10}
 \sum_{\lambda\in \mathbb Z^d} 2\pi i\lambda c(\lambda)e^{2\pi i\langle \lambda,x\rangle},
 \end{equation}
 obtained by applying the gradient operator
$\left(\frac{\partial}{\partial x_1},\frac{\partial}{\partial x_2},\cdots, \frac{\partial}{\partial x_d}\right)$
 term-by-term to the formal series \eqref{1}.
 Observe that
 \begin{equation}\label{11}
\lambda =\sum_{r=1}^d \langle v_r,\lambda\rangle  v_r^* 
 \end{equation}
where $v_1^*,\dots,v_d^*$ is the system of vectors biorthogonal to $v_1,\dots, v_d$ satisfying
$$\langle v_k,v_r^*\rangle=\begin{cases} 1, & k=r\\ 0, & k\ne r.\end{cases} $$
Using \eqref{11} and \lemref{lem3.2} we get that \eqref{10}
is the Fourier series of the vector-valued measure
\begin{equation}\label{11a}
\sum_{r=1}^d \Big( \pm \frac{D}{D_r} v_r^* \sum_j\sigma_{rj} \Big) -D\Big(\sum_{r=1}^d q_rv_r^*\Big)dx,
\end{equation}
where the sign $\pm$ is equal to the sign of $q_r$, and the summation with respect to $j$ is taken for $0\le j< q_r$
if $q_r>0,$ for $q_r\le j<0$ if $q_r<0,$ and understood to be zero if $q_r=0.$

A more convenient expression for this measure can be obtained if we define a measure
\begin{equation*}
\sigma:=\sum_{r=1}^d\sum_j\sigma_{rj},
\end{equation*}
where the summation with respect to $j$ is understood as in \eqref{11a},
and let $N(x)$ be the vector-valued function on the support of $\sigma$ given by
\begin{equation*}
N(x):={+}\frac{v_r^*}{\|v_r^*\|},\quad x\in \Pi_{rj}, \quad q_r>0,
\end{equation*}
\begin{equation*}
N(x):=-\frac{v_r^*}{\|v_r^*\|},\quad x\in \Pi_{rj}, \quad q_r<0.
\end{equation*}
Note that on each parallelepiped $\Pi_{rj}$, the vector $N(x)$ is a unit normal vector to this parallelepiped.
We also define the vector
\begin{equation}\label{16}
\omega:=D \sum_{r=1}^dq_rv_r^*.
\end{equation}
Using the fact that $\|v_r^*\|={D_r}/{D}$,
the measure in \eqref{11a} thus may be written as
$$N(x)\,d\sigma(x)-\omega\, dx.$$
This can be summarized by the following

\begin{lemma}\label{lem3.3}
The vector-valued series
\begin{equation}\label{17}
\sum_{\lambda\in\mathbb Z^d}2\pi i\lambda c(\lambda)e^{2\pi i\langle \lambda,x\rangle}
\end{equation}
is the Fourier series of the vector-valued measure 
\begin{equation}\label{18}
(\mu_1,\dots,\mu_d)=N(x)\,d\sigma(x)-\omega\, dx
\end{equation}  
on the torus $\mathbb T^d.$
\end{lemma}

It follows from \lemref{lem3.3} that the measures $\mu_1,\dots,\mu_d$ in \eqref{18} satisfy the conditions
\begin{equation}\label{20}
\int_{\mathbb T^d}d\mu_k=0\qquad (1\le k\le d)
\end{equation}
and
\begin{equation}\label{21}
\frac{\partial \mu_k}{\partial x_r}=\frac{\partial \mu_r}{\partial x_k}\qquad (k\ne r),
\end{equation}
where \eqref{21} is understood in the sense of distributions.
This is immediate from the fact that the Fourier series of $(\mu_1,\dots,\mu_d)$ is of the form \eqref{17},
that is, $\mu_1,\dots,\mu_d$ are formally the partial derivatives of the series \eqref{1}.

We remark that the converse is also true: if $\mu_1,\dots,\mu_d$ are measures (or distributions) on $\mathbb T^d$ satisfying \eqref{20} and \eqref{21}, then $\mu_1,\dots,\mu_d$ are the partial derivatives of some formal Fourier series (which can be shown to represent a distribution on $\mathbb T^d).$

With \lemref{lem3.3} established, we shall no longer need to refer to the specific form of the 
parallelepiped $P$ and the parallelepipeds $\Pi_{rj}$ associated to it.
The rest of the argument essentially relies on \eqref{20} and \eqref{21} only.

\subsection{Hypersurface of singularity}
We shall denote by $\Pi$ the collection of parallelepipeds $\{\Pi_{rj}\}$ (where $1\leq r \leq d$ and, as before, $0\le j< q_r$ if $q_r>0,$ and $q_r\le j<0$ if $q_r<0$) equipped with the orientation determined by the unit normal vector $N(x)$. We think of $\Pi$ as a piecewise-linear, oriented hypersurface on $\mathbb T^d$. This hypersurface may be self-intersecting, and overlapping parts of the $\Pi_{rj}$ are counted with the corresponding multiplicity. Such an object is standard in homology theory, and is called a ``$(d-1)$-chain''. Thus the measure $\sigma$ is the $(d-1)$-dimensional volume measure on $\Pi$.

\begin{lemma}\label{lem3.4}
The oriented hypersurface $\Pi$ has no boundary, that is, $\Pi$ is a closed hypersurface.
\end{lemma}

This means that the oriented boundaries of the parallelepipeds $\{\Pi_{rj}\}$ fit together in such a way
that the resulting hypersurface $\Pi$ has no boundary. In the language of homology theory this means that
$\Pi$ is a ``cycle''. \lemref{lem3.4} may be verified directly, using the explicit definition
\eqref{eq:pirj} of the parallelepipeds $\{\Pi_{rj}\}$ constituting $\Pi$. However, the proof below shows
that this is a more general fact which follows from \eqref{21}.

\begin{proof}[Proof of \lemref{lem3.4}]
We will prove the assertion by showing that
$$\int_{\partial \Pi} \eta=0$$
for every smooth differential $(d-2)$-form $\eta$ on $\mathbb T^d.$

Observe first that if $\phi$ is a smooth $(d-1)$-form given by
$$\phi=\varphi(x)dx_1\wedge\cdots d\hat{x}_m\dots\wedge dx_d,$$
where $d\hat{x}_m$ means that the term $dx_m$ is omitted, then
\begin{equation}\label{25}
\int_\Pi \phi=(-1)^{m-1}\int_{\mathbb T^d} \varphi(x) N_m(x) d\sigma(x),
\end{equation}
where $N_m(x)$ is the $m$'th coordinate of the vector $N(x)$.

Now suppose that $\eta$ is a smooth $(d-2)$-form. By Stokes' theorem
\begin{equation}\label{26}
\int_{\partial \Pi}\eta=\int_\Pi d\eta,
\end{equation}
so it will be enough to prove that the right-hand side of \eqref{26} vanishes. It is enough
to consider the case when
\begin{equation}\label{27}
\eta=\varphi(x)dx_1\wedge\cdots d\hat{x}_k\cdots d\hat{x}_m\cdots\wedge dx_d
\end{equation}
where $1\le k<m\le d,$ since any $(d-2)$-form may be expressed as the sum of basic $(d-2)$-forms as in \eqref{27}. In this case, we have
$$d\eta=(-1)^m \frac{\partial \varphi}{\partial x_m}dx_1\wedge\cdots d\hat{x}_k\cdots\wedge dx_d-(-1)^k \frac{\partial \varphi}{\partial x_k}dx_1\wedge\cdots d\hat{x}_m\cdots\wedge dx_d,$$
and hence by \eqref{25} we get
\begin{align*}
\int_\Pi d\eta=(-1)^{k+m} \left\{\int_{\mathbb T^d}\frac{\partial\varphi}{\partial x_k}N_m \, d\sigma-\int_{\mathbb T^d}\frac{\partial\varphi}{\partial x_m}N_k \, d\sigma\right\}.
\end{align*}
By \eqref{18} and since the partial derivatives of $\varphi$ have zero integral on $\bt^d$, this implies
\begin{align*}
\int_\Pi d\eta
= (-1)^{k+m} \left\{\int_{\mathbb T^d}\frac{\partial\varphi}{\partial x_k}d\mu_m-\int_{\mathbb T^d}\frac{\partial\varphi}{\partial x_m}d\mu_k\right\}
=(-1)^{k+m} \big\langle \frac{\partial \mu_k}{\partial x_m}-\frac{\partial\mu_m}{\partial x_k},\varphi\big\rangle
\end{align*}
which vanishes according to \eqref{21} and so yields the desired conclusion.
\end{proof}

Let $\gamma$ be a smooth oriented curve in $\mathbb T^d$, connecting two points $a$ and $b$ lying outside of the hypersurface $\Pi$. We assume that $\gamma$ is in ``general position'' in the sense that it intersects $\Pi$ only at interior points of the parallelepipeds $\{\Pi_{rj}\}$, and that at each intersection point the curve $\gamma$ is transversal to $\Pi$. To each point $x$ in $\Pi\cap\gamma$ one can then associate a sign $+1$ if $\gamma$ crosses $\Pi$ in the direction of the normal vector $N(x),$ or $-1$ if $\gamma$ crosses $\Pi$ in the opposite direction. The sum of all these signs is called the \emph{intersection number} of $\Pi$ and $\gamma$ and will be denoted by $\#(\Pi\cdot\gamma)$.

\begin{lemma}\label{lem3.5}
If $\gamma$ is a closed, oriented curve in $\mathbb T^d$ (in general position) then
\begin{equation}\label{30}
\#(\Pi\cdot\gamma)=\int_\gamma\omega_1dx_1+\dots+\omega_dx_d
\end{equation}
where $\omega=(\omega_1,\dots,\omega_d)$ is the vector from \eqref{18}.
\end{lemma}

\begin{proof}
Suppose that the lifting to $\mathbb R^d$ of the curve $\gamma$ is connecting the point $a$ to the point $b$, which (since $\gamma$ is closed) necessarily satisfy $b-a\in\mathbb Z^d.$ Since the hypersurface $\Pi$ is closed, the number $\#(\Pi\cdot\gamma)$ depends on the difference $b-a$ only
(this is a well-known fact in homology theory, that the intersection number depends only on the
homology classes of $\Pi$ and $\gamma$, see e.g.\ \cite[Section 0.4]{grhar}). The same is true for the right-hand side of \eqref{30}, which is equal to $\langle b-a,\omega\rangle$.

Since both sides of \eqref{30} are additive with respect to concatenation of curves, it will be enough to prove \eqref{30} in the case when $\gamma$ is one of the ``basic cycles''
$$\tau_k: t\mapsto t  e_k,\quad 0\le t\le 1,$$
where $e_1,\dots,e_d$ denote the standard basis vectors in $\mathbb R^d.$ Let
$x+\tau_k$ denote the curve
$$t\mapsto x+t  e_k,\qquad 0\le t\le 1,$$
obtained by translating $\tau_k$ along $x.$ We then have
$$\#(\Pi\cdot(x+\tau_k))=\#(\Pi\cdot\tau_k),\quad x\in\mathbb T^d ,$$
and hence
$$\#(\Pi\cdot\tau_k) =\int_{\mathbb T^d}\#(\Pi\cdot(x+\tau_k)) dx=\int_{\mathbb T^d}  N_k(x) \, d\sigma(x),$$
where $N_k(x)$ is the $k$'th coordinate of the vector $N(x)$. By \eqref{18} and \eqref{20}, we have
$$\int_{\mathbb T^d} N_k(x)\, d\sigma(x)=\int_{\mathbb T^d} \omega_k\, dx=\omega_k=\int_{\tau_k}\omega_1dx_1+\cdots+\omega_ddx_d.$$
This establishes \eqref{30} for $\gamma=\tau_k$ $(1\leq k\leq d)$ and thus proves the lemma.
\end{proof}

\begin{remark}
As the proof shows, we have $\omega_k=\#(\Pi\cdot\tau_k)$, and hence
$\omega$ is a vector in $\mathbb Z^d$ representing the hypersurface $\Pi$ in the $(d-1)$'th homology group of $\bt^d$.
\end{remark}

\subsection{Construction of the transfer function}
Consider a function $g(x)$ defined on $\mathbb T^d\setminus \Pi$ by
\begin{equation}\label{31}
g(x)=\#(\Pi\cdot\gamma(x))-\int_{\gamma(x)}\omega_1dx_1+\cdots +\omega_ddx_d,
\end{equation}
where $\gamma(x)$ is any curve (in general position) connecting some fixed point $x_0$ to the point $x.$
It follows from \lemref{lem3.5} that the right-hand side of \eqref{31} does not depend on the particular choice of the curve $\gamma(x).$
The function $g(x)$ is linear on each connected component of $\mathbb T^d\setminus \Pi$ (and in particular is continuous there), and presents a jump discontinuity on $\Pi.$

\begin{lemma}\label{lem3.6}
We have
$$\frac{\partial g}{\partial x_k}=\mu_k\quad (1\le k\le d)$$
in the sense of distributions.
\end{lemma}

\begin{proof}
It is enough to prove that given any point $a\in\mathbb T^d$, there is an open neighborhood $U$ of $a$ such that
\begin{equation}\label{35}
\frac{\partial g}{\partial x_k}=\mu_k \quad \text{in $U$.}
\end{equation}
If we choose $U$ to be a sufficiently small ball around the point $a$, then the (possibly self-intersecting) hypersurface $\Pi$ can be represented in $U$ as a finite union of non self-intersecting hypersurfaces $\Pi^{(i)}$, such that
each one of the $\Pi^{(i)}$ divides $U$ into two connected components. We denote by 
$U_+^{(i)}$ the connected component of $U \setminus \Pi^{(i)}$ for which the normal vector 
$N(x)$ on $\Pi^{(i)}$ is inward-pointing, and by $U_-^{(i)}$ the other component, having $N(x)$ outward-pointing.
Thus, we have
$$g(x)=\sum_i{\1}_{U_+^{(i)}}(x)-\langle x,\omega\rangle+c,\quad x\in U$$
for an appropriate constant $c.$

Now let $\varphi$ be a smooth function with compact support contained in $U.$ Then
\begin{equation}\label{32}
\big\langle\frac{\partial g}{\partial x_k},\varphi\big\rangle=-\int_{\mathbb T^d}\frac{\partial\varphi}{\partial x_k}(x)  g(x)dx=-\sum_i\int_{U_+^{(i)}}\frac{\partial \varphi}{\partial x_k}(x)dx+\int_U\frac{\partial \varphi}{\partial x_k}(x)\langle x,\omega\rangle dx
\end{equation}
(the constant term $c$ does not appear since the integral of $\partial \varphi/\partial x_k$ over $U$ vanishes).
The part common to the support of $\varphi$ and to the boundary of $U_+^{(i)}$ lies in $\Pi^{(i)},$ and $N(x)$ is the inward-pointing normal to $U_+^{(i)}$ on this part. Hence, by the divergence theorem,
\begin{equation}\label{33}
-\int_{U_+^{(i)}}\frac{\partial \varphi}{\partial x_k}(x)\,dx=\int_{\Pi^{(i)}}\varphi(x) N_k(x)\, d\sigma(x),
\end{equation}
where $N_k(x)$ is the $k$'th coordinate of the vector $N(x)$.
Also, integration by parts yields
\begin{equation}\label{34}
\int_U\frac{\partial \varphi}{\partial x_k}(x) \langle x,\omega\rangle dx = - \omega_k \int_{\mathbb T^d}\varphi(x) \, dx,
\end{equation}
again since $\varphi$ is supported in $U$.
Combining \eqref{32}, \eqref{33} and \eqref{34}, we get 
\begin{equation}\label{34a}
\big\langle\frac{\partial g}{\partial x_k},\varphi\big\rangle=\int\varphi(x)d\mu_k(x).
\end{equation}
As this holds for any smooth function $\varphi$ with compact support contained in $U$, this proves \eqref{35}.
\end{proof}

\subsection{Conclusion of the proof of \thmref{thm:main1specific}}
We have thus constructed a Riemann integrable (and, in particular, bounded) function $g(x)$ satisfying
\begin{equation}\label{36}
\left(\frac{\partial g}{\partial x_1},\dots,\frac{\partial g}{\partial x_d}\right)=(\mu_1,\dots,\mu_d)
\end{equation}
in the sense of distributions.
Hence both sides of \eqref{36} have the same Fourier series. But the Fourier series of the left-hand side of \eqref{36} is
$$\sum_{\lambda\in\mathbb Z^d} 2 \pi i\lambda\widehat g(\lambda)e^{2\pi i\langle \lambda,x\rangle},$$
while the Fourier series of the right-hand side is the series \eqref{17}. This shows that we must have
$$\widehat g(\lambda)=c(\lambda),\qquad \lambda \in\mathbb Z^d\setminus\{0\},$$
and this completes the proof of \thmref{thm:main1specific}.
\qed

\subsection{Zonotopes of bounded remainder}
With \thmref{thm:main1specific}
established, we can now deduce Corollary \ref{cor:zonotope}, which says that any zonotope in $\br^d$ with vertices in $\bz \va + \bz^d$ is a bounded remainder set. Recall again that a zonotope is a convex polytope which can be represented as the Minkowski sum of several line segments.
Equivalently, a zonotope is a convex, centrally symmetric polytope with centrally symmetric $k$-dimensional faces for $2 \leq k \leq  d-1$
(see e.g.\ \cite[Section 7.3]{ziegler}). Thus in $\br^2$ the zonotopes are the convex, centrally symmetric polygons. 

\begin{proof}[Proof of Corollary \ref{cor:zonotope}]
It is known that any zonotope $S$ can be tiled by a finite number of {\piped}s, and one can see from the proof that if the vertices of $S$ lie in $\bz \va + \bz^d$, then all the {\piped}s in this tiling are spanned by vectors in $\bz \va + \bz^d$ (\cite{shephard}, see also \cite[Theorem 2.48, p. 50]{concini}). Thus by Theorem \ref{thm:main1}, each of these {\piped}s is a BRS. Since the union of finitely many disjoint (up to measure zero) bounded remainder sets is also a BRS, Corollary \ref{cor:zonotope} follows.
\end{proof}

\subsection{Parallelepipeds of bounded remainder with pre-given measure}\label{sec:simplepiped}
Recall (Proposition \ref{prop:mesgiven}) that the measure $\gamma$ of any BRS must be of the form 
\begin{equation}
\gamma = n_0 + n_1 \alpha_1 + \cdots + n_d \alpha_d \quad (n_j \in \bz). 
\label{eq:themes}
\end{equation}
Here we prove Corollary \ref{cor:main1anymes}, which states that conversely, for any positive number $\gamma$ of the form \eqref{eq:themes} there exists a {\piped} $P$ of bounded remainder, spanned by vectors belonging to $\bz \va + \bz^d$, with $\mes P = \gamma$. Moreover, we show that if $\gamma \leq 1$ then $P$ can be chosen to be a simple set.

As Theorem \ref{thm:main1} is already proved, it remains only to show
\begin{proposition}
\label{prop:Pgivenmes}
For any positive number $\gamma$ of the form \eqref{eq:themes} there is a {\piped} $P$ in $\br^d$, spanned by vectors $v_1, \ldots , v_d $ belonging to $\bz \va + \bz^d$, such that $\mes P = \gamma$. If $\gamma \leq 1$, then $P$ may be chosen to be a simple set.
\end{proposition}
\begin{proof}
We assume $d  \geq 2$ (the one-dimensional case is trivial). Write $\gamma= q\ip{\va}{m} + r$, where $q,r \in \bz$ and $m \in \bz^d$ is a nonzero vector such that the gcd of its nonzero entries is $1$. Then one can choose vectors $m_1, \ldots, m_d \in \bz^d$ satisfying $\det(m_1, \ldots , m_d) = 1$ and such that $m_d=m$ (see e.g.\ \cite[Corollary 4, p.\ 14]{cassels}). Let $p_1, \ldots , p_d$ be the biorthogonal system satisfying $\ip{m_i}{p_j} = \delta_{ij}$. Then we have $p_1, \ldots , p_d \in \bz^d$ and $\det (p_1, \ldots , p_d) = 1$. Let $P$ be the {\piped} spanned by the vectors $p_1, \ldots, p_{d-1}$ and $q \va + rp_d$. Since $\va=\sum_{j=1}^d \ip{\va}{m_j} p_j$ it follows that $\det (p_1, \ldots , p_{d-1}, \va) = \ip{\va}{m}$, and hence
\begin{equation*}
\mes P = \det (p_1, \ldots , p_{d-1}, q \va + rp_d) = q \ip{\va}{m} + r  = \gamma.
\end{equation*}

Now suppose that $\gamma \leq 1$. Observe that a set $S$ is simple if and only if $(S-S) \cap \bz^d = \{ 0 \}$. A point $z \in P-P$ is of the form 
\begin{equation*}
z = \sum_{j=1}^{d-1} t_j p_j + t_d(q \va + rp_d), \quad -1 < t_j < 1 .
\end{equation*}
If, in addition, $z \in \bz^d$, then
\begin{equation*}
\bz \ni \ip{z}{m_d} =  t_d \ip{q\va+rp_d}{m}  =  t_d \gamma,
\end{equation*}
and since $0 < \gamma \leq 1$, it follows that $t_d = 0$. Similarly, for each $j=1, \ldots , d-1$
we have $t_j = \ip{z}{m_j} \in \bz$, and thus $t_j = 0$. We conclude that $z=0$. Hence, $P$ is simple. 
\end{proof}

\subsection{Generalization of Theorem \ref{thm:main1}}
\label{sec:genpipedstate}
Actually we can give a somewhat more general construction of bounded remainder {\piped}s, that can be formulated as follows.

\begin{theorem}
\label{thm:genpiped}
Let $v_1, \ldots , v_d \in \bz \va + \bz^d$, and let $P$ be a {\piped} spanned by vectors $w_1, \ldots , w_d$ satisfying
\[
w_1 = v_1, \quad w_k \in v_k + \spann \{ v_1, v_2, \dots, v_{k-1} \} \quad (2\leq k \leq d).
\]
Then $P$ is a BRS. 
\end{theorem}

For example, this includes Sz\"{u}sz' construction
of bounded remainder parallelograms in two dimensions (\thmref{thm:szusz}). It can
be obtained by choosing two vectors $v_1 \in \bz \va + \bz^2$ and $v_2 \in \bz^2$, and taking $P$ to be the parallelogram
spanned by the vectors $w_1 = v_1$ and $w_2 = v_2+tv_1$, where $t \in \br$ is chosen such that the 
vector $w_2$ lies on the $x$-axis.

Since we derive \thmref{thm:genpiped} from Theorem \ref{thm:main1} using
the notion of equidecomposability, we postpone its proof to the next \sect\ref{sec:main2}.


\section{Equidecomposability of bounded remainder sets}\label{sec:main2}

\subsection{}
The main goal of this section is to obtain a characterization of the Riemann measurable bounded remainder 
sets in terms of equidecomposability. Recall that if $G$ is a group of motions of the space $\br^d$, then
two measurable sets $S$ and $S'$ are said to be \emph{$G$-equidecomposable} if the set $S$ can be partitioned 
into finitely many measurable subsets that can be reassembled by motions of the group $G$ to form, up to measure zero, a partition of $S'$.
If $S$ and $S'$ belong to a restricted class of sets, e.g.\ they are Riemann measurable sets, or they are polytopes, then we require the pieces of the partition to belong to the same class.

We first show that if two sets $S$ and $S'$ are equidecomposable with respect to the group of
translations by vectors in $\bz \va + \bz^d$, and if $S$ is a bounded remainder set, then so is $S'$.
This is basically known, see e.g.\ \cite[p. 277]{liardet}.

Then we prove the main result of this section, Theorem \ref{thm:main2}, which states that 
any two Riemann measurable bounded remainder sets of the same measure are equidecomposable
(by Riemann measurable pieces) using translations by vectors belonging to $\bz \va + \bz^d$ only.
The equidecomposition is constructed by an explicit iterative procedure.
In the case when the given sets are polytopes, this procedure yields an equidecomposition 
using pieces which are polytopes as well.

Combining this with Theorem \ref{thm:main1} and \corref{cor:main1anymes} thus yields
a characterization of the Riemann measurable bounded remainder sets,
in terms of equidecomposability to a {\piped} spanned by vectors in $\bz \va + \bz^d$
(Corollary \ref{cor:equidecomp}).

Finally, we show that every Riemann measurable BRS admits a Riemann integrable transfer function
(Theorem \ref{thm:Rtransfer}), and also give the proof of \thmref{thm:genpiped} stated above.

\subsection{}
We turn to the details, starting with the following
\begin{proposition}
\label{prop:shiftsubset}
Let $S$ and $S'$ be two bounded, measurable sets in $\br^d$. Suppose that $S$ and $S'$ 
are equidecomposable using only translations by vectors in $\bz \va + \bz^d$.
If $S$ is a bounded remainder set, then so is $S'$.
\end{proposition}

\begin{proof}
We assume that $S$ may be partitioned into a finite number of measurable subsets $S_j$, and that each $S_j$
may be translated by a vector $\gamma_j \in \bz \va + \bz^d$, such that the translated sets $S_j' := S_j + \gamma_j$
form, up to measure zero , a partition of $S'$. We also assume that $S$ is a BRS,
so there is a function $g \in L^{\infty}(\bt^d)$ such that
\begin{equation*}
\chi_S(x) - \mes S = g(x) - g(x-\alpha) \quad \text{a.e.}
\end{equation*}
Write $\gamma_j = n_j \va + m_j$ where $n_j \in \bz$ and $m_j \in \bz^d$. If $n_j > 0$, observe that the function
\begin{equation*}
g_j(x) := \sum_{k=0}^{n_j-1} \chi_{S_j}(x - k\alpha) 
\end{equation*}
satisfies
\begin{equation}
\label{eq:diffsub}
\chi_{S_j}(x) - \chi_{S_j'}(x) = g_j(x) - g_j(x-\alpha).
\end{equation}
If $n_j < 0$ one can define a function $g_j$ satisfying \eqref{eq:diffsub} similarly, while in the case $n_j= 0$ we just take $g_j =0$.
Since $S$ and $S'$ have the same measure, it follows that the function $g'(x) := g(x) - \sum g_j(x)$ satisfies
\begin{equation}
\chi_{S'}(x) - \mes S' = g'(x) - g'(x-\alpha) \quad \text{a.e.,}
\end{equation}
that is, $g'$ is a bounded transfer function for $S'$. Hence, $S'$ is a BRS.
\end{proof}

\subsection{Auxiliary lemmas}
The next two lemmas will be needed
since we are working with sets in $\br^d$ which are not necessarily simple sets. 
\begin{lemma}
\label{lemma:subset}
Let $A \subset \br^d$ be a bounded, Riemann measurable set, and suppose that $\varphi : \bt^d \rightarrow \bz$ is a Riemann integrable function satisfying
\begin{equation}
\label{eq:subset}
0 \leq \varphi(x) \leq \chi_A(x), \quad \text{ a.e. } x \in \bt^d .
\end{equation}
Then there exists a Riemann measurable subset $A' \subset A$ for which $\chi_{A'} = \varphi$ a.e.
\end{lemma}

\begin{proof}
Consider $Q=[0,1)^d$ as a representative for $\bt^d$. Since $A$ is bounded, there exist distinct vectors $m_1, \ldots , m_M \in \bz^d$ such that $A= \bigcup (A_i + m_i)$, where
\begin{equation*}
A_i := (A-m_i) \cap Q , \quad 1 \leq i \leq M .
\end{equation*}
Restrict the function $\varphi$ to $Q$, and define sets $S_1, \ldots , S_M \subset Q$ by their indicator functions
\begin{equation*}
\1_{S_1} := \min \{ \varphi ,  \1_{A_1} \} , 
\end{equation*}
and by induction
\begin{equation*}
\1_{S_i} := \min \{ \varphi - \1_{S_1} - \ldots - \1_{S_{i-1}}, \1_{A_i} \}, \quad i=2, 3, \ldots , M .
\end{equation*}
It is easy to check, using \eqref{eq:subset}, that
\begin{equation}
\sum_{i=1}^{M} \1_{S_i}(x) = \varphi (x) \quad \text{ a.e. } x \in Q. 
\label{eq:phiequal}
\end{equation}
Hence 
$A' := \bigcup_{i=1}^M (S_i + m_i)$
is the union of disjoint sets contained in $A$, and it follows from \eqref{eq:phiequal} that we have $\chi_{A'}(x) = \varphi (x)$ a.e., as required.
\end{proof}

\begin{lemma}
\label{lemma:equalproj}
Let $A$ and $B$ be two bounded, Riemann measurable sets in $\br^d$ such that 
\begin{equation}
\chi_A(x) = \chi_B(x), \quad \text{ a.e. } x \in \bt^d .
\label{eq:equalproj}
\end{equation}
Then $A,B$ are equidecomposable (by Riemann measurable pieces) using only
translations by vectors in $\bz^d$.
\end{lemma}

\begin{proof}
Since $A$ and $B$ are bounded, there exist distinct vectors $m_1, \ldots , m_M \in \bz^d$ such that $A, B \subset \bigcup (Q + m_i)$, where $Q=[0,1)^d$. Define $A^1 := A \cap (Q+m_1)$. We have that
\begin{equation*}
\chi_{A^1} \leq \chi_A = \chi_B  \quad \text{ a.e.},
\end{equation*}
so by Lemma \ref{lemma:subset} there is a subset $B^1 \subset B$ for which $\chi_{A^1} = \chi_{B^1}$ a.e. 
Subtract $\chi_{A^1} = \chi_{B^1}$ from both sides in \eqref{eq:equalproj} to obtain 
\begin{equation*}
\chi_{A \setminus A^1}(x) = \chi_{B \setminus B^1}(x), \quad \text{ a.e. } x \in \bt^d,
\end{equation*}
and repeat this procedure for $A^i := A \cap (Q+m_i)$ with $i=2, \ldots , M$. With each iteration we find a subset $B^i \subset B$ for which $\chi_{B^i} = \chi_{A^i}$ a.e. At the $M$'th step both $A$ and $B$ are exhausted,
and we obtain two partitions $A = \bigcup_1^M A_i$ and $B = \bigcup_1^M B_i$ up to measure zero. Finally we split each $A^i$ further by letting
\begin{equation*}
A_j^i = A^i \cap (B^i + m_i- m_j) , \quad 1 \leq j \leq M .
\end{equation*}
Then (since $\chi_{A^i} = \chi_{B^i}$ a.e.) we obtain
\begin{equation*}
A = \bigcup_{i,j = 1}^M A_j^i \quad \text{ and } \quad B = \bigcup_{i,j=1}^M \left( A_j^i + k_j^i \right) \quad
\end{equation*}
up to measure zero, where $k_j^i := m_j -m_i \in \bz^d$.
\end{proof}

\subsection{Main lemma}
The next lemma is the key step in the proof of Theorem \ref{thm:main2}.

\begin{lemma}
Let $A$ and $B$ be bounded, Riemann measurable sets in $\br^d$ such that 
\begin{equation}
\chi_A(x) - \chi_B(x) = g(x) - g(x-\alpha) \quad \text{ a.e. } x \in \bt^d ,
\label{eq:diffbds}
\end{equation} 
where $g$ is a measurable, non-negative function on $\bt^d$. Assume that 
\begin{equation}
\chi_A(x) \cdot \chi_B(x+k\alpha) = 0 \quad \text{a.e. for } 0 \leq k \leq n-1,
\label{eq:ncondition}
\end{equation}
and let $A' \subset A$ be a Riemann measurable set such that 
\begin{equation}
\chi_{A'}(x) = \min \left\{ \chi_A(x), \chi_{B}(x+n\alpha) \right\}  \quad \text{a.e.} 
\label{eq:aprime}
\end{equation}  
We then have 
\begin{equation}
g'(x) := \sum_{k=0}^{n-1} \chi_{A'}(x-k \alpha)  \leq g(x) \quad \text{ a.e. } 
\label{eq:g1}
\end{equation}
\label{lemma:g1}
\end{lemma}

\begin{remark}
Observe that condition \eqref{eq:diffbds} implies (by integration) that $A$ and $B$ have the same measure. Furthermore, Lemma \ref{lemma:subset} ensures that a set $A' \subset A$ satisfying \eqref{eq:aprime} indeed exists. It also ensures the existence of a set $B' \subset B$ for which $\chi_{A'}(x) = \chi_{B'}(x+n\alpha)$ a.e. Notice that the function $g'$ defined in \eqref{eq:g1} is then a transfer function for the difference $\chi_{A'} - \chi_{B'}$, that is, 
\begin{equation*}
\chi_{A'}(x)-\chi_{B'}(x) = g'(x) - g'(x- \alpha) \quad \text{ a.e. } x \in \bt^d .
\end{equation*}
The lemma then says that under the above conditions, this new transfer function is dominated by the given transfer function $g$.
This fact will not enter in the proof of Lemma \ref{lemma:g1}, but it will be used later on, in the proof of Theorem \ref{thm:main2}.
\end{remark}

\begin{proof}[Proof of Lemma \ref{lemma:g1}]
Let $S$ denote the image of $A'$ under the canonical projection $\br^d \rightarrow \bt^d$. Then $\chi_{A'}$ is supported by $S$. Since $g \geq 0$ it is clear from the definition \eqref{eq:g1} of $g'$ that it suffices to show 
\begin{equation}
g'(x+m\alpha) \leq g(x+m\alpha) , \quad \text{ a.e.\ on } S , \quad 0 \leq m \leq n-1. 
\label{eq:whatweshow}
\end{equation}
Fix the integer $m$. First we claim that 
\begin{equation}
g'(x+m\alpha) = \chi_{A'}(x) \quad \text{a.e.\ on } S .
\label{eq:step1}
\end{equation}
This follows from the fact that the functions $\chi_{A'}(x-k\alpha)$, $k=0,1, \ldots , n-1$, have a.e.\ disjoint supports. Indeed, by \eqref{eq:ncondition} and \eqref{eq:aprime}, we have 
\begin{equation*}
\chi_{A'}(x-k\alpha) \cdot \chi_{A'}(x-j\alpha) \leq \chi_A(x-k\alpha) \cdot \chi_B(x+(n-j)\alpha) = 0 \quad \text{a.e.},
\end{equation*}
when $0 \leq k < j \leq n-1$. Using the definition of $g'$ this implies \eqref{eq:step1}.

Secondly, we argue that 
\begin{equation}
g(x+m\alpha)=g(x) \quad \text{a.e.\ on } S . 
\label{eq:step2}
\end{equation}
This is because both terms on the left hand side in \eqref{eq:diffbds} are zero a.e.\ on $S+k\alpha$ for $k=1, \ldots , n-1$. Indeed $\chi_B(x+k\alpha) = 0$ a.e.\ on $S$, as
\begin{equation*}
\chi_{B}(x+k\alpha) \cdot \chi_{A'}(x) \leq \chi_B(x+k\alpha) \cdot \chi_A(x) = 0 \quad \text{a.e. }
\end{equation*}
We also have that $\chi_{A}(x+k\alpha)= 0$ a.e. on $S$, since
\begin{equation*}
\chi_{A}(x+k\alpha) \cdot \chi_{A'}(x) \leq \chi_A(x+k\alpha) \cdot \chi_B(x+n\alpha) = 0 \quad \text{a.e. }
\end{equation*}
From \eqref{eq:diffbds} it then follows that 
\begin{equation*}
g(x+m\alpha) = g(x+(m-1)\alpha) = \cdots = g(x+\va)=g(x) \quad \text{a.e.\ on } S.
\end{equation*}

Lastly, observe that 
\begin{equation}
\chi_{A'}(x) \leq \chi_{A}(x) + g(x-\alpha) = g(x) \quad \text{ a.e.\ on } S, 
\label{eq:step3}
\end{equation}
where the inequality is true since $g \geq 0$, while the equality follows from the cohomological equation \eqref{eq:diffbds}. Combining \eqref{eq:step1}, \eqref{eq:step2} and \eqref{eq:step3} we obtain \eqref{eq:whatweshow}. 
\end{proof}

\subsection{Proof of Theorem \ref{thm:main2}}
Let $A$ and $B$ be two Riemann measurable bounded
remainder sets of the same measure. Then there exists a bounded, measurable function $g$ on $\bt^d$ such that
\begin{equation}
\chi_A(x) - \chi_B(x) = g(x) - g(x-\alpha) \quad \text{ a.e. } x \in \bt^d .
\label{eq:diffbdsref}
\end{equation}
Indeed, $g$ is the difference of the transfer functions for $A$ and $B$. By adding an appropriate constant to $g$, we may assume that $g$ is non-negative. 

We partition $A$ and $B$ simultaneously in the following inductive manner. First let $A^0 \subset A$ be a set with 
\begin{equation*}
\chi_{A^0}(x) = \min \{ \chi_A(x), \chi_B(x) \}  \quad \text{a.e.}
\end{equation*}
Lemma \ref{lemma:subset} ensures the existence of such a set, and also guarantees that there is a subset $B^0 \subset B$ with $\chi_{B^0}(x) = \chi_{A^0}(x)$ a.e. We observe that the difference $\chi_{A \setminus A^0}(x) - \chi_{B \setminus B^0}(x)$ satisfies the cohomological equation \eqref{eq:diffbdsref} with the same function $g(x)$.

We proceed (again by Lemma \ref{lemma:subset}) by finding a set $A^1 \subset A \setminus A^0$ such that
\begin{equation*}
\chi_{A^1}(x) = \min \{ \chi_{A \setminus A^0}(x), \chi_{B \setminus B^0}(x+\alpha) \} \quad \text{a.e.} 
\end{equation*}
Since the sets $A \setminus A^0$ and $B \setminus B^0$ satisfy \eqref{eq:ncondition}
for $n=1$, we conclude from Lemma \ref{lemma:g1} that
\begin{equation*}
g_1(x) := \chi_{A^1}(x) \leq g(x) \quad \text{a.e.}
\end{equation*}
We also find a set $B^1 \subset B \setminus B^0$ for which $\chi_{B^1}(x+ \alpha)=\chi_{A^1}(x)$ a.e., and observe that $g-g_1$ is a non-negative transfer function for the difference $\chi_{A \setminus (A^0 \cup A^1)} - \chi_{B \setminus (B^0 \cup B^1)}$. 

We then continue in the same way. At the $n$'th step in the iteration we have two sets, $A' = A \setminus \cup_{i=0}^{n-1} A^i$ and $B' = B \setminus \cup_{i=0}^{n-1} B^i$, which satisfy the cohomological equation
with non-negative transfer function $g-\sum_{i=1}^{n-1} g_i$. We find a subset $A^n \subset A'$ such that 
\begin{equation*}
\chi_{A^n}(x)=\min \{  \chi_{A'}(x), \chi_{B'}(x+n\alpha) \} \quad \text{a.e.}
\end{equation*}
Since the sets $A'$ and $B'$ satisfy \eqref{eq:ncondition}, we can apply Lemma \ref{lemma:g1} and conclude that 
\begin{equation}
g_n(x) := \sum_{j=0}^{n - 1} \chi_{A^n}(x-j\alpha) 
\label{eq:defgn}
\end{equation}
satisfies $g_n(x) \leq g(x)-\sum_{i=1}^{n-1} g_i(x)$ a.e. 
We also find a subset $B^n \subset B'$ such that
\begin{equation}
\chi_{A^n}(x)=\chi_{B^n}(x+n \alpha) \quad \text{a.e.}
\label{eq:anbn}
\end{equation}

This process yields an infinite sequence of disjoint subsets $A^n$ of $A$ (respectively $B^n$ of $B$).
We claim that in fact, the sets $A^n$ (respectively $B^n$) exhaust all of $A$ (respectively of $B$) up
to measure zero. Indeed, since $\mes A^n = \mes  B^n$, the sets
$U = A \setminus \cup_{n=0}^{\infty} A^n$ and $V = B \setminus \cup_{n=0}^{\infty} B^n$
have the same measure and satisfy
\begin{equation*}
\chi_U(x) \cdot \chi_V(x+k\alpha) = 0 \quad \text{a.e. } \quad k=1,2,3, \ldots 
\end{equation*}
Since the points $ \{ k \va \}$ are dense in $\bt^d$,
this is possible only if $U,V$ have measure zero.

Next we claim that actually only finitely many of the sets $A^n$ can have positive measure. Suppose to the contrary that there are infinitely many $A^n$ for which $\mes A^n >0$. All the sets $A^n$ are Riemann measurable, so each $A^n$ must contain a ball. Since the finite set $\{ j \alpha \}$ ($0 \leq j \leq n-1$) is $\varepsilon$-dense in $\bt^d$ for all sufficiently large $n \geq n(\alpha, \varepsilon)$, and due to the definition \eqref{eq:defgn} of the function $g_n$, we can find an increasing sequence $n_j$ and a sequence of balls $\mathcal{B}_j$, such that $\mathcal{B}_{j+1} \subset \mathcal{B}_{j}$ and $g_{n_j}(x) \geq 1$ a.e.\ on $\mathcal{B}_j$. It follows that 
the sum $\sum_{i=1}^{n} g_i$ admits, with positive measure, arbitrarily large values as $n \to \infty$. 
But $\sum_{i=1}^n g_i(x) \leq g(x)$ a.e, and $g$ is a bounded function, so this is a contradiction.

We thus obtain a finite partition $A = \cup A^n$ a.e.\ and a corresponding partition $B = \cup B^n$ a.e., where $A^n$ and $B^n$ are related by \eqref{eq:anbn}. Finally, by using Lemma \ref{lemma:equalproj}, we partition each $A^n$ further into finitely many sets $A^n_i$, and find
corresponding vectors $k^n_i \in \bz^d$, such that setting $B^n_i := A^n_i + n \alpha + k^n_i$ we have
\begin{equation*}
A = \bigcup_{n,i} A^n_i \quad \text{ and } \quad B = \bigcup_{n,i} B^n_i
\end{equation*}
up to measure zero. This completes the proof of Theorem \ref{thm:main2}.
\qed

\subsection*{Remarks}\label{sec:remarksmain2} 
(i) Observe that we have actually constructed a partition of $A$ which is reassembled to obtain $B$ using only translations by vectors in
$\bz^+ \va + \bz^d$ (where $\bz^+$ is the set of non-negative integers).

(ii) The pieces of the partition are constructed by a finite number of
translations, intersections and differences, starting from the sets $A$, $B$ and the unit cube $Q$.
It follows that if $A$ and $B$ are polytopes, then the pieces of the partition constructed will be polytopes as well.

\subsection{Characterization of Riemann measurable bounded remainder sets}
The characterization given by Corollary \ref{cor:equidecomp} now follows.

\begin{proof}[Proof of Corollary \ref{cor:equidecomp}]
Suppose first that a Riemann measurable set $S \subset \br^d$ is equidecomposable 
to some {\piped} $P$ spanned by vectors in $\bz \va + \bz^d$, using 
translations by vectors belonging to $\bz \va + \bz^d$.
By Theorem \ref{thm:main1}, $P$ is a BRS, so it follows from Proposition \ref{prop:shiftsubset} that $S$ is a BRS as well.

Conversely, suppose that $S$ is a Riemann measurable BRS.
By Proposition \ref{prop:mesgiven} and Corollary \ref{cor:main1anymes}
there exists a {\piped} $P$ of bounded remainder, spanned by vectors belonging to $\bz \va + \bz^d$, with $\mes P = \mes S$. 
Theorem \ref{thm:main2} implies that $S$ is equidecomposable to $P$ using only translations by vectors belonging to $\bz \va + \bz^d$.
\end{proof}

\subsection{Riemann integrability of the transfer function}
\label{sec:Rtransfer}
We continue to show that a Riemann measurable BRS admits a Riemann integrable transfer function,
thus proving Theorem \ref{thm:Rtransfer}.

This relies essentially on the observation that (i) the transfer function for a {\piped} spanned by vectors in $\bz \va + \bz^d$,
constructed in the proof of Theorem \ref{thm:main1}, is Riemann integrable; and (ii) the equidecomposition
of two Riemann measurable bounded remainder sets of the same measure, constructed in the proof of 
Theorem \ref{thm:main2}, consists of Riemann measurable pieces.

\begin{proof}[Proof of Theorem \ref{thm:Rtransfer}]
Let $S$ be a Riemann measurable BRS. By Corollary \ref{cor:equidecomp}, $S$ is equidecomposable (with Riemann
measurable pieces) to some {\piped} $P$ spanned by vectors in $\bz \va + \bz^d$,
using translations by vectors belonging to $\bz \va + \bz^d$. By \thmref{thm:main1specific},
$P$ is a bounded remainder set, and moreover has a Riemann integrable transfer function. 
The proof of Proposition \ref{prop:shiftsubset} then shows that $S$ has a Riemann integrable transfer function
as well (it is the difference of the transfer function for $P$ and the function $\sum g_j(x)$ from the proof
of Proposition \ref{prop:shiftsubset}).
\end{proof}

\subsection{Proof of \thmref{thm:genpiped}}
\label{sec:genpipedproof}
Now we can also give a proof of \thmref{thm:genpiped} formulated in \sect\ref{sec:main1} above. It is based on the following
\begin{lemma}
Let $P$ be a {\piped} spanned by vectors $v_1, \ldots , v_d$, and suppose that for some $j$
we have $v_j \in \bz \va + \bz^d$. Let $P'$ be another {\piped}, spanned by the vectors
$v_1, \ldots , v_{k-1} , v_k + sv_j , v_{k+1}, \ldots , v_d$, where $s \in \br$ and $k \neq j$.
Then $P$ and $P'$ are equidecomposable using only translations by vectors in $\bz \va + \bz^d$.
\label{lem:genpiped}
\end{lemma}
\begin{proof}
Suppose first that $0 < s \leq 1$. In this case we partition the {\piped} $P$ into two disjoint subsets $S_1$ and $S_2$ defined by
\begin{equation*}
S_1 = \left\{ \sum_{i=1}^d t_i v_i \, : \, 0 \leq t_j < st_k, \;  0 \leq t_i < 1 \; (i \neq j) \right\}
\end{equation*}
and
\begin{equation*}
S_2 = \left\{ \sum_{i=1}^d t_i v_i \, : \, st_k \leq t_j < 1, \; 0 \leq t_i < 1  \; (i \neq j) \right\}.
\end{equation*}
It is easy to verify that the two sets $S_1 + v_j$ and $S_2$ constitute a partition of $P'$.
Since $v_j \in \bz \va + \bz^d$, this proves the claim.

Since equidecomposability (with respect to any group of motions) is an equivalence relation, 
the result can now be extended to the case $s > 1$ by repeating the previous argument several times,
and to the case $s<0$ by exchanging the roles of $P$ and $P'$.
\end{proof}

\begin{proof}[Proof of Theorem \ref{thm:genpiped}]
Let $P_0$ be the {\piped} spanned by the vectors $v_1, \ldots , v_d$.
By iteratively applying Lemma \ref{lem:genpiped} we conclude 
(again, since equidecomposability is an equivalence relation)
that $P$ is equidecomposable to $P_0$
using translations by vectors in $\bz \va + \bz^d$ only. Since $P_0$ is a BRS (by Theorem \ref{thm:main1})
it follows from Proposition \ref{prop:shiftsubset} that $P$ is also a BRS.
\end{proof}


\section{Polytopes of bounded remainder and Hadwiger invariants}
\label{sec:hadwiger}
In this section we use the characterization of Riemann measurable bounded remainder sets to study polytopes of bounded remainder. The approach is based on the concept of \emph{additive invariants}. 

Recall that if $G$ is a group of motions of the space $\br^d$, then a function $\varphi$ defined on the set of all polytopes is said to be an additive $G$-invariant if (i) it is additive, namely if $S_1, S_2$ are two polytopes with disjoint interiors then $\varphi(S_1 \cup S_2) = \varphi(S_1) + \varphi(S_2)$; and (ii) it is invariant under motions of the group $G$, that is $\varphi(S) = \varphi(g(S))$ whenever $S$ is a polytope and $g \in G$. It is clear that a necessary condition for two polytopes $S$ and $S'$ to be $G$-equidecomposable 
is that $\varphi(S) = \varphi(S')$ for any additive $G$-invariant $\varphi$.

Additive invariants with respect to the group of all translations of $\br^d$ were introduced by Hadwiger, see \cite{hadwiger2, hadwigerbook}. It was proved that these invariants form a complete set, in the sense that together they provide a necessary and sufficient condition for two polytopes of the same volume to be equidecomposable by translations. This was shown by Hadwiger and Glur in dimension two \cite{glur}, by Hadwiger in dimension three \cite{hadwiger}, and by Jessen and Thorup \cite{jessen}, and independently Sah \cite{sah}, in any dimension. 

In this section we first define Hadwiger-type invariants with respect to an arbitrary subgroup of all the translations of $\br^d$. 

Using the characterization of Riemann measurable bounded remainder sets, in terms of equidecomposability with the respect to the group of translations by vectors in $\bz \va + \bz^d$, we thus obtain necessary conditions for a polytope in $\br^d$ to be a bounded remainder set (Theorem \ref{thm:hadwigerzero}). 

This is then used to analyze certain specific cases. In particular, we consider finite unions of intervals in one dimension, and convex polytopes in dimensions two and higher.

\subsection{Flags and additive weight functions}
Fix an integer $0 \leq k \leq d-1$, and let
\begin{equation*}
V_k \subset V_{k+1} \subset \cdots \subset V_{d-1} \subset V_d = \br^d
\end{equation*}
be a sequence of affine subspaces such that $V_j$ has dimension $j$ (by an affine subspace we mean a translated linear subspace). Each subspace $V_j$ ($k \leq j \leq d-1$) divides $V_{j+1}$ into two half-spaces; let us call one of them the positive half-space, and the other the negative half-space. Such a sequence, consisting of affine subspaces and positive/negative half-spaces, will be called a $k$-flag, and will be denoted by $\Phi$ (to motivate the name ``flag", imagine a 0-flag in three dimensions). 

Let $S$ be a polytope in $\br^d$. Suppose that $S$ has a sequence of faces 
\begin{equation*}
F_k \subset F_{k+1} \subset \cdots \subset F_{d-1} \subset F_d =S ,
\end{equation*}
where $F_j$ is a $j$-dimensional face contained in $V_j$ for each $j= k, \ldots , d-1$. To each face $F_j$ we associate a coefficient $\varepsilon_j$, where $\varepsilon_j =+1$ if $F_{j+1}$ adjoins $V_j$ from the positive side, and $\varepsilon_j = -1$ if $F_{j+1}$ adjoins $V_j$ from the negative side. We then define the ``weight function"
\begin{equation}
\label{eq:weight}
\omega_{\Phi}(S) = \sum \varepsilon_{k} \varepsilon_{k+1} \cdots \varepsilon_{d-1} \vol_k (F_k) ,
\end{equation}
where the sum runs though all sequences of faces of $S$ with the above mentioned property, and where $\vol_k$ stands for the $k$-dimensional volume. In particular, if no such sequences of faces of $S$ exist, then $\omega_{\Phi}(S) = 0$. Remark that a $0$-dimensional face is simply a vertex $p$ of $S$, and $\vol_0(p)=1$. The function $\omega_{\Phi}$ is then an additive function on the set of all polytopes in $\br^d$ (see e.g.\ \cite[Section 2.6]{sah}). 

It should be mentioned that here, a ``polytope" is not assumed to be convex, nor even connected.
Thus, a polytope may be understood as any finite union of $d$-dimensional simplices with disjoint interiors. 

\subsection{Hadwiger-type invariants}
Let $\Gamma$ be an arbitrary subgroup of $\br^d$. For each $k$-flag $\Phi$ we now define an additive invariant $H_{\Phi}$ with respect to the group of translations by vectors in $\Gamma$. This is done by considering the sum of weights
\begin{equation}
H_{\Phi}(S) = H_{\Phi}(S, \Gamma) = \sum_{\Psi} \omega_{\Psi}(S) , 
\label{eq:hadwiger}
\end{equation}
where $\Psi$ runs through all distinct $k$-flags such that $\Psi = \Phi + \gamma$ for some $\gamma \in \Gamma$. Notice that only finitely many terms in the sum can be nonzero, as the number of nonzero terms is limited by the number of $k$-dimensional faces of $S$. The function $H_{\Phi}$ is easily seen to be an additive invariant, and will be called the Hadwiger invariant associated to $\Phi$. If $\Phi$ is a $k$-flag, then we will say that $H_{\Phi}$ is an invariant of rank $k$. 

Notice that if two $k$-flags $\Phi$ and $\Psi$ correspond to the same sequence of affine subspaces $V_k \subset \cdots \subset V_d = \br^d$, then the Hadwiger invariants $H_{\Phi}$ and $H_{\Psi}$ are either equal, or differ by a factor $-1$. Thus, each sequence of affine subspaces essentially provides one Hadwiger invariant. For this reason we will not always specify the choices of positive and negative half-spaces of the subspaces $V_j$ in what follows. 

If $\Gamma= \br^d$, then the invariants $H_{\Phi}$ are precisely the classical invariants proposed by Hadwiger. The case where $\Gamma$ is a proper subgroup of $\br^d$, however, seems to be less explored. In this case we do not know whether equality of Hadwiger invariants is not just a necessary, but also a sufficient, condition for two polytopes of equal volume to be equidecomposable using translations by vectors in $\Gamma$.

We remark that in the classical case $\Gamma = \br^d$, $0$-rank invariants have not been considered, as they vanish identically and thus do not provide any information. To the contrary, when $\Gamma$ is a proper subgroup of $\br^d$, $0$-flags do provide nontrivial invariants. 

\subsection{Bounded remainder polytopes and Hadwiger invariants}
Now we consider the case when $\Gamma=\bz \va + \bz^d$. Using the characterization of the Riemann measurable bounded remainder sets, we can use Hadwiger invariants to give explicit necessary conditions for a polytope to be a BRS. 

\begin{theorem}
\label{thm:hadwigerzero}
For a polytope $S$ in $\br^d$ to be a bounded remainder set, it is necessary that
\begin{equation*}
H_{\Phi}(S, \, \bz \va + \bz^d)=0
\end{equation*} 
for any $k$-flag $\Phi$ $(0 \leq k \leq d-1)$.
\end{theorem}
\begin{proof}
If $S$ is a bounded remainder set, then by Corollary \ref{cor:equidecomp} it is equidecomposable to a {\piped} $P$, spanned by vectors in $\bz \va + \bz^d$, using translations by vectors in $\bz \va + \bz^d$. Hence, $H_{\Phi}(S)=H_{\Phi}(P)$ for any $k$-flag $\Phi$ ($0 \leq k \leq d-1$). We claim that $H_{\Phi}(P) = 0$ for any such $k$-flag. To see this, observe that for any sequence $F_k \subset F_{k+1} \subset \cdots \subset F_d$ of faces of $P$, there is a unique $k$-dimensional face $F_k'$ of $P$ such that $F_k' \subset F_{k+1} \subset \cdots \subset F_d$ and $F_k' = F_k + \gamma$ for some $\gamma \in \bz \va + \bz^d$. Hence, if the sequence of faces containing $F_k$ contributes to $H_{\Phi}(P)$ for some $k$-flag $\Phi$, then the sequence containing $F_k'$ will contribute equally, but with opposite sign, to $H_{\Phi}(P)$. Thus, we must have $H_{\Phi}(P)=0$, and therefore $H_{\Phi}(S)=0$. 
\end{proof}
It might be that the condition given in Theorem \ref{thm:hadwigerzero} is also sufficient for $S$ to be a bounded remainder set. We do not attempt to prove this in general. However, we will see that this is indeed true in certain special cases. 

\subsection{Dimension one. Finite unions of intervals}
Let us first discuss what Theorem \ref{thm:hadwigerzero} says in dimension one. The polytope $S$ is then just a finite union of disjoint intervals $[a_j, b_j]$, and a $0$-flag $\Phi$ is simply a point $p$, dividing $\br$ into a positive and a negative part. The corresponding Hadwiger invariant $H_{\Phi}$ sums up the number of endpoints $a_j$, with one sign, and $b_j$, with the opposite sign, contained in the orbit $\{ p+\gamma \, : \, \gamma \in \bz \va + \bz \}$. Theorem \ref{thm:hadwigerzero} states that $H_{\Phi}(S) = 0$ for any chosen point $p$. Hence, for a finite union of intervals to be a set of bounded remainder it is necessary that any orbit contains an equal number of left and right endpoints. This is a consequence of a result due to Oren \cite[Theorem A]{oren} which shows that the latter condition in fact characterizes the finite unions of intervals with bounded remainder.
\begin{theorem}[Oren]
Let $S \subset \br$ be the union of $N$ disjoint intervals $[ a_j , b_j ]$, $1 \leq j \leq N$. Then $S$ is a bounded remainder set if and only if there exists a permutation $\sigma$ of $\{1, \ldots , N\}$ such that
\begin{equation}
b_{\sigma (j)} - a_j \in \bz \va + \bz  \quad ( 1 \leq j \leq N) .
\label{eq:oren}
\end{equation}
\label{thm:oren}
\end{theorem}

Thus, the necessity part in this result is a consequence of Theorem \ref{thm:hadwigerzero}. For completeness of the exposition, we also include a proof of the (easier) sufficiency part. 

\begin{proof}[Proof of the sufficiency part in Theorem \ref{thm:oren}]
Suppose that there exists a permutation $\sigma$ satisfying \eqref{eq:oren}. For each $1 \leq j \leq N$, denote by $\varphi_j$ the function $\1_{[a_j, b_{\sigma(j)}]}$ if $a_j < b_{\sigma(j)}$, or $-\1_{[b_{\sigma(j)}, a_j]}$ if $a_j > b_{\sigma(j)}$. The function $\sum_{j=1}^N \varphi_j$ then has the following properties: it has a jump discontinuity of magnitude $+1$ at each $a_j$ and of magnitude $-1$ at each $b_j$, it is constant between these jumps, and it vanishes off the interval $[\min S, \max S]$. This determines $\sum \varphi_j$ uniquely a.e.\ as the indicator function $\1_S$. Hence, $\1_S$ is a finite linear combination (with coefficients $\pm 1$) of indicator functions of bounded remainder intervals, due to \eqref{eq:oren} and the Hecke-Ostrowski Theorem \ref{thm:hecke}. It follows that $S$ must be a BRS.
\end{proof}

\subsection{Dimension two. Convex polygons}
Consider now the case when $S$ is a convex polygon in $\br^2$. If $S$ is a bounded remainder set, then by Theorem \ref{thm:hadwigerzero} we have $H_{\Phi}(S)=0$ for every Hadwiger invariant $H_{\Phi}$ of rank $0$ or $1$. We will see that this condition in fact characterizes the convex polygons of bounded remainder, and that this is an equivalent formulation of Theorem \ref{thm:cvxpg}.

Let us first check that the vanishing of Hadwiger invariants is equivalent to $S$ being centrally symmetric and satisfying conditions \ref{item:facets1} and \ref{item:facets2} in Theorem \ref{thm:cvxpg}.

Let $e$ be an edge of $S$, and consider the $1$-flag $\Phi$ defined by the line $l$ containing $e$. The set $S$ is convex, so there is at most one other edge $e'$ parallel to $e$. The condition $H_{\Phi}(S)=0$ guarantees that there is indeed such an edge $e'$, and the length of $e'$ equals that of $e$. Since this holds for any edge $e$, $S$ must be centrally symmetric. Furthermore, for every pair of parallel edges $e$ and $e'$ there must exist a vector $\gamma \in \bz \va + \bz^2$ such that $e' \subset l + \gamma$, where $l$ is the line containing $e$. 

Now let again $e$ be an edge of $S$, and let $p$ be one of the endpoints of $e$. Consider the $0$-flag $\Phi$ defined by the point $p$ and the line $l$ containing $e$. The condition $H_{\Phi}(S)=0$ implies that if the other endpoint of $e$ does not belong to the orbit $\{p + \gamma \, : \, \gamma \in \bz \va + \bz^2 \}$, then this orbit must contain the unique endpoint $p'$ of $e'$ whose contribution to the sum \eqref{eq:hadwiger} would cancel that of $p$. This is illustrated in Figure \ref{fig:parallel}, and implies condition \ref{item:facets2} in Theorem \ref{thm:cvxpg}. Combined with the fact that $e' \subset l + \gamma$ for some $\gamma \in \bz \va + \bz^2$, this also implies condition \ref{item:facets1}. 
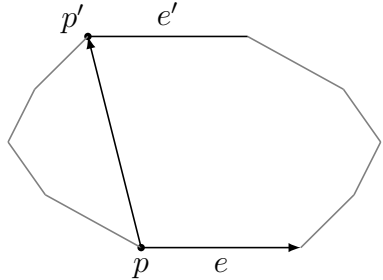
\begin{figure}[htb]
\centering
\begin{tikzpicture}[scale=0.7, p1/.style={black!60!white}, p2/.style={black}, p3/.style={black!50!white}, semithick]
\fill (0,0) circle (0.07);
\draw (-.85,3.85) node[above left]{$p'$};
\fill (-1,4) circle (0.07);
\draw (0,0) node[below]{$p$};
\draw[-latex, p2] (0, 0) -- (3, 0);
\draw (1.5, 0) node[below]{$e$};
\draw[p3] (3,0) -- (4,1);
\draw[p3] (4,1) -- (4.5,2);
\draw[p3] (4.5,2) -- (3.8, 3);
\draw[p3] (3.8,3) -- (2, 4);
\draw[p2] (2, 4) -- (-1,4);
\draw (0.5, 4) node[above]{$e'$};
\draw[p3] (-1,4) -- (-2, 3);
\draw[p3] (-2,3) -- (-2.5, 2);
\draw[p3] (-2.5, 2) -- (-1.8, 1);
\draw[p3] (-1.8,1) -- (0,0);
\draw[-latex, p2] (0,0) -- (-1,4);
\end{tikzpicture}
\caption{A convex polygon $S$ with parallel edges $e$ and $e'$. Condition \ref{item:facets2} says that $e$ or $p'-p$ must lie in $\bz \va + \bz^2$.}
\label{fig:parallel}
\end{figure}

In a similar way one can see that conversely, if $S$ is centrally symmetric and satisfies \ref{item:facets1} and \ref{item:facets2}, then it has vanishing rank $0$ and $1$ Hadwiger invariants. An equivalent formulation of Theorem \ref{thm:cvxpg} is thus the following.

\begin{theorem}
\label{thm:cphadwiger}
Let $S$ be a convex polygon in $\br^2$. Then $S$ is a bounded remainder set if and only if $H_{\Phi}(S, \, \bz \va + \bz^2)=0$ for all $0$- and $1$-flags $\Phi$.
\end{theorem}

We now turn to the proof of this theorem. 

\begin{proof}[Proof of Theorem \ref{thm:cphadwiger}]
By Theorem \ref{thm:hadwigerzero}, the vanishing of all rank $0$ and $1$ invariants is a necessary condition for $S$ to be a BRS, so only the proof of the sufficiency remains. We therefore suppose that $H_{\Phi}(S) = 0$ for all $0$- and $1$-flags $\Phi$, and prove that $S$ is a BRS. 

As $S$ is a centrally symmetric polygon, it can be represented as the Minkowski sum of several line segments. By translating $S$, we may thus assume that it is of the form 
\begin{equation*}
S = \left\{ \sum_{i=1}^{n} t_i v_i \, : \, 0 \leq t_i < 1 \right\},
\end{equation*}
where the vectors $v_1, \ldots , v_n$ denote $n$ consecutive edges, enumerated in counterclockwise order, among a total of $2n$ edges of $S$. This is illustrated in Figure \ref{fig:decagon}. The proof is done by induction on $n$. 

In the case $n=2$, $S$ is a parallelogram spanned by two vectors $v_1$ and $v_2$. By conditions \ref{item:facets1} and \ref{item:facets2} in Theorem \ref{thm:cvxpg}, one of these vectors, say $v_1$, belongs to $\bz \va + \bz^2$, while the other vector satisfies $v_2 + tv_1 \in \bz \va + \bz^2$ for some $t \in \br$. 
Thus, $S$ is a BRS by Theorem \ref{thm:genpiped}.

Now assume that $n>2$. If all the vectors $v_1, \ldots , v_n$ lie in $\bz \va+ \bz^2$, then $S$ is a BRS by Corollary \ref{cor:zonotope}. So suppose this is not the case, and without loss of generality assume that $v_n \notin \bz \va + \bz^2$. Condition \ref{item:facets2} then implies that the midpoints of the two edges parallel to $v_n$ differ by a vector in $\bz \va + \bz^2$. This means that the vector $\gamma_1$, shown in Figure \ref{fig:decagon}, is in $\bz \va + \bz^2$. Again by \ref{item:facets1} and \ref{item:facets2}, one can find a point on the edge parallel to $v_1$ which differs from the origin by a vector $\gamma_2 \in \bz \va + \bz^2$ (again see Figure \ref{fig:decagon}). 
\begin{figure}[htb]
\centering
\begin{tikzpicture}[scale=0.7, p1/.style={black!60!white}, p2/.style={black}, p3/.style={black!50!white}, semithick]
\fill (0,0) circle (0.07);
\draw (0.1,0.1) node[below left]{\footnotesize 0};
\draw[p3] (0, 0) -- (3, 0);
\draw (1.5, 0) node[below]{$v_1$};
\draw[p3] (3,0) -- (4,1);
\draw (3.5, 0.5) node[below, right]{$v_2$};
\draw[p3] (4,1) -- (4.5,2);
\draw (4.25, 1.5) node[below, right]{$v_3$};
\draw[p3] (4.5,2) -- (3.8, 3);
\draw (4.15, 2.5) node[right]{$\cdots$};
\draw[p3] (3.8,3) -- (2, 4);
\draw (2.9, 3.5) node[above]{$v_n$};
\draw[p3] (2, 4) -- (-1,4);
\draw[p3] (-1,4) -- (-2, 3);
\draw[p3] (-2,3) -- (-2.5, 2);
\draw[p3] (-2.5, 2) -- (-1.8, 1);
\draw[p3] (-1.8,1) -- (0,0);

\draw[-latex, p2] (0,0) -- (3.8,3);
\draw (1.92, 1.48) node[p2, below, right]{$\gamma_1$};
\draw[-latex, p2] (0,0) -- (0.5,4); 
\draw (0.25, 2) node[p2, right]{$\gamma_2$};
\end{tikzpicture}
\caption{The convex polygon $S$ and the vectors $\gamma_1, \gamma_2 \in \bz \va + \bz^2$. \label{fig:decagon}}
\end{figure}

We now partition the polygon $S$ into five subsets $S_1, \ldots , S_5$, as illustrated in Figure \ref{fig:partition}, where the vector $w$ in the figure is defined by $w = v_n - \gamma_2$.
Consider the union $A=S_2 \cup S_3 \cup S_4$. It is equidecomposable (using translations by vectors in $\bz \va + \bz^2$) to the (disjoint) union $P = S_3 \cup (S_2 + \gamma_1) \cup (S_4 + \gamma_2)$. Notice that $P$ is a parallelogram spanned by the vectors $\gamma_1$ and $\gamma_2$, and hence it is a BRS (by Theorem \ref{thm:main1}). By Proposition \ref{prop:shiftsubset} it follows that $A$ is a BRS, which implies by Theorem \ref{thm:hadwigerzero} that $H_{\Phi}(A) = 0$ for any $0$- or $1$-flag $\Phi$. From the additivity of $H_{\Phi}$ it follows that also $H_{\Phi}(B)=0$, where $B =S_1 \cup S_5$. Notice that $B$ is equidecomposable to the disjoint union $S' = (S_1-\gamma_2) \cup (S_5 - \gamma_1)$, a convex polygon which is the Minkowski sum of $n-1$ line segments. We have $H_{\Phi}(S') = H_{\Phi}(B) = 0$, so by the induction hypothesis $S'$ is a BRS. Again by Proposition \ref{prop:shiftsubset} it follows that $B$ is a BRS. The set $S$ is thus a union of two disjoint bounded remainder sets $A$ and $B$. Hence, $S$ is a BRS. 
\begin{figure}[htb]
\centering
\begin{tikzpicture}[scale=1, p2/.style={black!50!white}, p1/.style={black}, p3/.style={red}, semithick]
\draw[p2] (0, 0) -- (3, 0);
\draw[p2] (3,0) -- (4,1);
\draw[p2] (4,1) -- (4.5,2);
\draw[p2] (4.5,2) -- (3.8, 3);
\draw[p2] (3.8,3) -- (2, 4);
\draw[p2] (2, 4) -- (-1,4);
\draw[p2] (-1,4) -- (-2, 3);
\draw[p2] (-2,3) -- (-2.5, 2);
\draw[p2] (-2.5, 2) -- (-1.8, 1);
\draw[p2] (-1.8,1) -- (0,0);

\draw[-latex, p1] (0,0) -- (3.8,3);
\draw (2, 1.4) node[p1, above left ]{$\gamma_1$};
\draw[-latex, p1] (0,0) -- (0.5,4);
 \draw (0.25, 2) node[p1, right]{$\gamma_2$};
\draw[-latex, p1] (3.8,3) -- (1.5, 0);
\draw (2.75, 1.5) node[p1, right, below]{$w$};
\draw[-latex, p1] (0.5, 4) -- (-1.8,1);
\draw (-0.55, 2.5) node[p1, right, below]{$w$};

\draw[p2] (-1.4, 2.6) node{$S_1$};
\draw[p2] (-0.5, 1.4) node{$S_2$};
\draw[p2] (1.4, 2.6) node{$S_3$};
\draw[p2] (1.4, 0.5) node{$S_4$};
\draw[p2] (3.5, 1.4) node{$S_5$};
\end{tikzpicture}
\caption{The partition of the convex polygon $S$.\label{fig:partition}}
\end{figure}
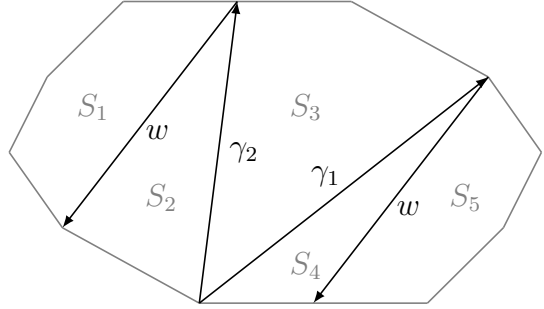
\end{proof}

\subsection{Dimension three and higher}
We have seen that the vanishing of Hadwiger invariants is a necessary and sufficient condition for $S$ to be a BRS in the case when $S$ is a finite union of intervals in dimension one, and when $S$ is a convex polygon in two dimensions. 
In higher dimensions, this condition is still necessary (Theorem \ref{thm:hadwigerzero}), but we do not know if it is sufficient even when the polytope $S$ is convex. 

The condition that $S$ is centrally symmetric, obtained for a convex polygon in $\br^2$, remains a necessary condition for a convex polytope in $\br^d$ to be a BRS, in any dimension $d$. Moreover, also the $(d-1)$-dimensional faces of $S$ must be centrally symmetric. This is the assertion of Theorem \ref{thm:cvxptpsym}, which we shall now prove:

\begin{proof}[Proof of Theorem \ref{thm:cvxptpsym}]
This follows from a result of M\"{u}rner \cite[Section 3.3]{murner}, who showed that for a convex polytope $S$, the vanishing of all the (classical) Hadwiger invariants $H_{\Phi}(S, \, \br^d)$ is equivalent to $S$ being centrally symmetric and having centrally symmetric $(d-1)$-dimensional faces. Observe that the condition $H_{\Phi}(S, \, \Gamma) = 0$ for all $k$-flags $\Phi$, where $\Gamma$ is a proper subgroup of $\br^d$, clearly implies that also $H_{\Phi}(S, \, \br^d) = 0$. If $S$ is a BRS, the former condition holds with $\Gamma = \bz \va + \bz^d$ by Theorem \ref{thm:hadwigerzero}, and so the result follows.  
\end{proof} 

Of course, the condition that $H_{\Phi}(S, \, \bz \va + \bz^d) = 0$ for all $k$-flags $\Phi$, gives more information. For example, if $H$ is the hyperplane containing a $(d-1)$-dimensional face of $S$, then the parallel face must be contained in $H+ \gamma$ for some $\gamma \in \bz \va + \bz^d$.

Recall that we also gave a sufficient condition for a convex polytope to be a BRS,
namely that it is a zonotope with vertices belonging to $\bz \va + \bz^d$ (\corref{cor:zonotope}).
In dimension $d \geq 4$, the class of convex, centrally symmetric polytopes with centrally symmetric $(d-1)$-dimensional faces,
is strictly larger than the class of all zonotopes. However, if $d=3$ then these two classes coincide. We thus obtain
that a convex polyhedron in $\br^3$ with vertices belonging to $\bz \va + \bz^3$
is a BRS if and only if it is a zonohedron, namely it is centrally symmetric and has centrally symmetric faces (\corref{cor:zonohedron}).

We mention another necessary condition which follows from Theorem \ref{thm:hadwigerzero}, and which will be useful later on (here, the polytope $S$ need not be convex). 

\begin{theorem}\label{thm:vertices}
Let $S$ be a polytope in $\br^d$. If $S$ is a bounded remainder set, then any vertex of $S$ differs by a vector in $\bz \va + \bz^d$ from at least one other vertex. 
\end{theorem}
\begin{proof}
Let $p$ be a vertex of $S$. Choose a sequence $p=F_0 \subset \cdots \subset F_d=S$ of faces of $S$, and let $\Phi$ be a $0$-flag defined by the sequence of affine subspaces $V_0 \subset \cdots \subset V_d$ satisfying $F_j \subset V_j$ ($0 \leq j \leq d$). We then have $\omega_{\Phi}(S) = \pm 1$. Now consider the Hadwiger invariant $H_{\Phi}(S)$. Since $H_{\Phi}(S) = 0$ by Theorem \ref{thm:hadwigerzero}, there is at least one $0$-flag $\Psi = \Phi + \gamma$, where $0 \neq \gamma \in \bz \va + \bz^d$, such that $\omega_{\Psi}(S) = -\omega_{\Phi}(S)$. This means that there exists a sequence of faces $p'=F_0' \subset \cdots \subset F_d'=S$ such that $F_j' \subset V_j + \gamma$ for each $j=0, \ldots , d$. In particular, the two vertices $p$ and $p'$ satisfy $p' = p + \gamma$, and this confirms the claim.
\end{proof}

\subsection{Rectangles with sides parallel to the coordinate axes.}
We conclude this section by demonstrating how another result, due to Liardet \cite[Theorem 3]{liardet}, may also be deduced from Theorem \ref{thm:hadwigerzero}. The result characterizes the bounded remainder
multi-dimensional rectangles with sides parallel to the coordinate axes:

\begin{theorem}[Liardet]
\label{thm:liardetblocks}
If $S \subset \br^d$ is the product of $d$ intervals $I_1 \times \cdots \times I_d$ then $S$ is a bounded remainder set if and only if the length of one of the intervals $I_j$ belongs to $\bz \va_j + \bz$, while the lengths of all the other intervals belong to $\bz$.
\end{theorem}

\begin{proof} 
We assume $d \geq 2$.
The sufficiency of the condition follows easily from the Hecke-Ostrowski Theorem \ref{thm:hecke},
so we move on to prove the necessity. Suppose that $S$ is a BRS. 
By translating $S$ we may assume that $I_j = [0, l_j)$, where $l_j > 0$ for each $j=1, \ldots, d$.

We first show that at least $(d-1)$ of the intervals $I_j$ must be of integer length. This is equivalent to saying that among any two values $l_i$ and $l_j$ $(i \neq j)$, at least one is an integer. As the argument is the same for any pair, we show this for $l_1$ and $l_2$. Let $\Phi$ be a $(d-2)$-flag defined by the subspaces $V_{d-2} = \spann \{ e_3, \ldots, e_d \}$ and $V_{d-1} = \spann \{ e_2, \ldots , e_d \}$, where $ e_1, \ldots , e_d$ denote the standard basis vectors in $\br^d$. Since the two faces
\begin{equation*}
F_{d-2} = \{ 0 \} \times \{ 0\} \times I_3 \times \cdots \times I_d, \qquad 
F_{d-1} = \{ 0 \} \times I_2 \times I_3 \times \cdots \times I_d,
\end{equation*}
are the only pair contributing to the sum \eqref{eq:weight}, then by an appropriate choice of positive-negative half-spaces of the flag $\Phi$ we have $\omega_{\Phi}(S) = \vol_{d-2} (F_{d-2})$. 

Now consider the Hadwiger invariant $H_{\Phi}$, which by Theorem \ref{thm:hadwigerzero} vanishes on $S$. This implies that there is a $(d-2)$-flag $\Psi = \Phi + \gamma$ for some nonzero $\gamma \in \bz \va + \bz^d$, such that $\omega_{\Psi}(S)$ is negative. Hence, there is a pair of faces $F_{d-2}' \subset F_{d-1}'$ of $S$ such that $F_{d-2}' \subset V_{d-2} + \gamma$ and $F_{d-1}' \subset V_{d-1} + \gamma$, and this pair must be one among the following two possible pairs:

(i) $F_{d-2}' = F_{d-2} + l_2e_2$ and $F_{d-1}' = F_{d-1}$. In this case, $F_{d-2} + l_2e_2 \subset V_{d-2} + \gamma$ implies that the vector $\gamma - l_2e_2$ has vanishing first and second entries. The first entry is simply $\gamma_1 \in \bz \va_1 + \bz$, and since $\va_1$ is irrational this implies that $\gamma \in \bz^d$. The second entry is $\gamma_2 - l_2$, and since $\gamma \in \bz^d$ it follows that $l_2 = \gamma_2 \in \bz$.

(ii) $F_{d-2}' = F_{d-2}+l_1e_1$ and $F_{d-1}' = F_{d-1} + l_1 e_1$. In this case, it follows by a similar argument that $l_1 \in \bz$.

Thus we have shown that there is at most one $j$ for which $l_j$ is not an integer.

We now show that if $l_j \notin \bz$ for some $j$, then $l_j$ must belong to $\bz \va_j + \bz$. Again, the argument is the same for any $j$, so consider the case $j=1$. Let $\Phi$ be a $(d-1)$-flag defined by the subspace $V_{d-1} = \spann \{ e_2, \ldots e_d \}$. The face
$F_{d-1} = \{ 0 \} \times I_2 \times I_3 \times \cdots \times I_d$ is the unique one contained in $V_{d-1}$, and we may suppose that $\omega_{\Phi}(S) = \vol_{d-1} (F_{d-1})$. Since $H_{\Phi}(S) = 0$ (Theorem \ref{thm:hadwigerzero}), it follows that the parallel face $F_{d-1}' = F_{d-1} + l_1 e_1$ must satisfy $F_{d-1}' \subset V_{d-1} + \gamma$ for some $\gamma \in \bz \va + \bz^d$. This implies that the vector $\gamma - l_1 e_1$ has vanishing first entry, and therefore $l_1 = \gamma_1 \in \bz \va_1 + \bz$. This completes the proof.
\end{proof}


\section{Linear maps and bounded remainder sets}
\label{sec:linear}

In this section we study relations between bounded remainder sets which 
correspond to two different irrational vectors $\va$ and $\vb$ in $\br^d$.
Specifically, we consider the situation of an invertible linear map $T$ 
on $\br^d$, which maps bounded remainder 
sets with respect to $\va$ to bounded remainder sets with respect to $\vb$.

First we prove \thmref{thm:main3}, which provides a description of
 the linear maps $T$ which map every Riemann measurable BRS with respect to $\va$
to a BRS with respect to $\vb$. We show that this is the case if and only if
$T(\bz \va + \bz^d) \subset \bz \vb + \bz^d.$

In particular, this allows us to characterize the situation when 
every BRS with respect to $\va$ is also a BRS with respect to $\vb$,
by the condition $\va \in \bz \vb + \bz^d$
(\corref{cor:brsab}).

Then we explain how to construct all the irrational vectors $\vb$ and the
 invertible linear maps $T$ satisfying the condition
$T(\bz \va + \bz^d) \subset \bz \vb + \bz^d$ for a given irrational vector $\va$,
and show that
they can be parametrized by the $(d+1)\times (d+1)$ integer matrices with non-zero determinant
(\thmref{thm7.1}).

Finally, we prove by a different argument that $T$ maps every BRS with respect to $\va$
to a BRS with respect to $\vb$ if the stronger condition 
$T(\bz \va + \bz^d) = \bz \vb + \bz^d$ is satisfied (\thmref{thm7.2}).
The advantage of this proof is that it does not rely on the characterization 
of the Riemann measurable bounded remainder sets (i.e.\ \corref{cor:equidecomp}), and hence it is also valid for
bounded remainder sets which are not Riemann measurable.

\subsection{Proof of Theorem \ref{thm:main3}}

 Assume first that $T(\mathbb Z\alpha+\mathbb Z^d)\subset \mathbb Z\beta+\mathbb Z^d.$ We show that if $S$ is a Riemann measurable BRS with respect to $\alpha,$ then $T(S)$ is a BRS with respect to $\beta.$ Indeed, by Corollary \ref{cor:equidecomp}, $S$ is equidecomposable to a parallelepiped $P,$ spanned by vectors in $\mathbb Z\alpha+\mathbb Z^d,$ using translations by vectors in $\mathbb Z\alpha+\mathbb Z^d.$ Thus, $S$ may be partitioned into a finite number of Riemann measurable subsets $S_j$ such that if each $S_j$ is translated by an appropriate vector $\gamma_j\in\mathbb Z\alpha+\mathbb Z^d,$ then the translated sets $S_j+\gamma_j$ form, up to measure zero, a partition of $P.$

 It follows that the set $T(S)$ admits a partition into sets $T(S_j),$ such that the translated sets $T(S_j)+T(\gamma_j)$ form, up to measure zero, a partition of $T(P).$ The set $T(P)$ is a parallelepiped spanned by vectors in $\mathbb Z\beta+\mathbb Z^d,$ and the vectors $T(\gamma_j)$ lie in $\mathbb Z\beta +\mathbb Z^d.$ Hence, again by Corollary \ref{cor:equidecomp}, $T(S)$ is a BRS with respect to $\beta.$
This proves one part of the theorem.

Now we turn to prove the converse part. 
Assume that $T$ maps any Riemann measurable BRS with respect to $\alpha,$ to a BRS with respect to $\beta.$ We show that $T(\mathbb Z\alpha+\mathbb Z^d)\subset \mathbb Z\beta+\mathbb Z^d.$
In the one-dimensional case this follows easily from the Hecke-Ostrowski-Kesten characterization of
the bounded remainder intervals, so we consider the case $d \geq 2$.

 Suppose to the contrary that there is a vector $v\in\mathbb Z\alpha+\mathbb Z^d$ such that $Tv\notin \mathbb Z \beta+\mathbb Z^d.$ 
Choose any system of $d$ linearly independent vectors $v_1,\dots,v_d \in\mathbb Z\alpha+\mathbb Z^d$ with $v_1=v.$
 For each $t\in\mathbb R$, consider a parallelepiped $P_t$ spanned by the vectors $w_1(t),\dots,w_d(t)$ defined by
 $$w_1(t):=v_1,\qquad w_k(t):=v_k+t v_1\qquad (2\le k\le d).$$
 By Theorem \ref{thm:genpiped}, the (non-degenerate) parallelepiped $P_t$ is a BRS with respect to $\alpha.$ Hence its image $T(P_t)$ is a BRS with respect to $\beta.$

 The set $T(P_t)$ is also a parallelepiped, spanned by the vectors $Tw_1(t),\dots, Tw_d(t).$ Hence the vertices of $T(P_t)$ are the vectors of the form
 $$\sum_{j\in J} Tw_j(t),$$
 where $J$ goes through all subsets of $\{1,2,\dots,d\}.$ The vertex which corresponds to the empty subset lies at the origin,
and by Theorem \ref{thm:vertices} it must differ from at least one other vertex by a vector in $\mathbb Z \beta+\mathbb Z^d.$ We conclude that there exists a non-empty subset $J_t$ of $\{1,2,\dots,d\}$ such that 
 $$\sum_{j\in J_t} Tw_j(t) \in \mathbb Z\beta+\mathbb Z^d.$$
Since $T w_1(t) = Tv \notin \mathbb Z\beta+\mathbb Z^d$, we have $J_t \neq \{1\}.$

 Since there is only a finite number of possibilities for the set $J_t,$ there must exist a fixed set $J$ such that $J_t=J$ for all $t$ in some uncountable set $E\subset\mathbb R.$ For such $t$ we have
 $$\sum_{j\in J_t} Tw_j(t)=\sum_{j\in J} Tw_j(t)=a+th,$$
 where $$a=\sum_{j\in J}Tv_j$$
 and $h$ is the non-zero vector given by
$h=|J\setminus\{1\}|\cdot Tv$.
 It follows that $\mathbb Z\beta+\mathbb Z^d$ contains all the vectors of the form
$a+th$, $t\in E$,
 which is an uncountable set of vectors. But $\mathbb Z\beta+\mathbb Z^d$ is a countable set, so this yields the desired contradiction.
This completes the proof of Theorem \ref{thm:main3}.
 \qed

\begin{proof}[Proof of Corollary \ref{cor:brsab}]
By Proposition \ref{prop:bversusa}, the condition $\va \in \bz \vb + \bz^d$ implies that
any BRS with respect to $\va$ is also a BRS with respect to $\vb$. Conversely, using
\thmref{thm:main3} for the identity map shows that the condition $\va \in \bz \vb + \bz^d$
is also necessary.
\end{proof}

\subsection{Parametrization of the pairs $(\beta,T)$}
We now explain how to construct all the irrational vectors $\vb$ and the
 invertible linear maps $T$ on $\mathbb R^d$ satisfying the condition
$T(\bz \va + \bz^d) \subset \bz \vb + \bz^d$ for a given irrational vector $\va$.
We will show that
they can be parametrized by the $(d+1)\times (d+1)$ integer matrices with non-zero determinant.

 Observe that any such $\beta$ and $T$ induce a mapping $U$ on $\mathbb Z^d\times \mathbb Z$ defined by $U(m,n)=(m',n'),$ where $(m',n')$ is the unique element in $\mathbb Z^d\times\mathbb Z$ such that
 \begin{equation}\label{5.2}
 T(n\alpha+m)=n'\beta+m'.
 \end{equation}
 Since the mapping $U$ is additive, there is a $d\times d$ integer matrix $A,$ two vectors $p,q\in\mathbb Z^d$ and a scalar $r\in\mathbb Z,$ such that
  \begin{equation}\label{5.3}
m'=Am+pn,
 \end{equation}
   \begin{equation}\label{5.4}
n'=\langle q,m\rangle+rn.
 \end{equation}
 Thus $U$ may be identified with a $(d+1) \times (d+1)$ matrix with integer entries.

 We claim that $\det U \neq 0$.
Indeed, if not, there would necessarily exist a non-zero vector $(m,n)\in\mathbb Z^d\times\mathbb Z$ which is mapped by $U$ to $(0,0).$ 
By \eqref{5.2} this would imply that $T(n\alpha+m)=0.$
 But since $T$ is invertible and $\alpha$ is irrational, this is possible only if $(m,n)=(0,0),$ a contradiction.

 Next we claim that the map $T$ is given by
   \begin{equation}\label{5.5}
Tx=Ax+\langle q,x\rangle\beta
 \end{equation}
 for every $x\in\mathbb R^d.$ Indeed, it follows from \eqref{5.2}, \eqref{5.3} and \eqref{5.4} that
equality \eqref{5.5} is true for all integer vectors $x$.
 But since two linear maps which agree on integer vectors must agree everywhere, this implies that \eqref{5.5} indeed holds for every $x\in\mathbb R^d.$

It also follows from \eqref{5.2}, \eqref{5.3} and \eqref{5.4} that
$T\alpha=r\beta+p.$
Combining this with \eqref{5.5}, we arrive at the equality
   \begin{equation}\label{5.7}
\beta=\frac{A\alpha-p}{r- \langle q,\alpha\rangle}
 \end{equation}
(observe that the denominator is non-zero, as $\va$ is irrational and $\det U \neq 0$).
To conclude, we have shown that any irrational vector $\beta$ and invertible linear map $T$ satisfying 
$ T(\mathbb Z\alpha+\mathbb Z^d)\subset\mathbb Z\beta+\mathbb Z^d$
 are of the form \eqref{5.7} and \eqref{5.5}.

Now we shall verify that the converse statement is also true. Namely, let $A$ be a $d\times d$ integer matrix, $p,q$ be two vectors in $\mathbb Z^d,$ and $r$ be a scalar in $\mathbb Z$, such that the mapping $U:(m,n)\mapsto(m',n')$ given by \eqref{5.3}, \eqref{5.4} has $\det U \neq 0$. We will show 
that \eqref{5.7} and \eqref{5.5} indeed define an irrational vector $\beta$ 
and an invertible linear map $T$ such that
$T(\mathbb Z\alpha+\mathbb Z^d)\subset\mathbb Z\beta+\mathbb Z^d$.

First we show that the vector $\beta$ defined by \eqref{5.7} is irrational.
If not, there would exist a non-zero $(k,j)\in\mathbb Z^d\times\mathbb Z$ such that
$\langle\beta,k\rangle=j$.
 Using \eqref{5.7} this is equivalent to
 $$\langle\alpha,A^\top k+qj\rangle=\langle p,k\rangle+rj.$$
 Since $\alpha$ is irrational, this implies
 $$A^\top k+qj=0,$$
 $$\langle p,k\rangle+rj=0,$$
 that is, $U^\top (k,j)=0.$ But since we have $\det U \neq 0$, this is not possible unless $(k,j)=(0,0)$.
Hence $\beta$ is irrational.

 Now we let $T$ be the linear map defined by \eqref{5.5}. First observe that $T$ satisfies
\eqref{5.2}, and therefore we have
$T(\mathbb Z\alpha+\mathbb Z^d)\subset\mathbb Z\beta+\mathbb Z^d$.
Let us show that $T$ is invertible. Indeed, since the image of $\mathbb Z^{d+1}$ 
under $U$ contains $(\det U)\cdot\mathbb Z^{d+1},$
it follows from \eqref{5.2} that the image of $T$ contains
 $(\det U)\cdot(\mathbb Z\beta+\mathbb Z^d),$ a dense subset of $\mathbb R^d.$ Hence $T$ must be invertible.

To summarize the above, we have proved the following

 \begin{theorem}\label{thm7.1}
 Let $\alpha$ be an irrational vector in $\mathbb R^d$.
Let $A$ be a $d\times d$ integer matrix, $p,q$ be two vectors in $\mathbb Z^d,$ and $r$ be a scalar in $\mathbb Z$, such that the mapping $U:(m,n)\mapsto(m',n')$ given by \eqref{5.3}, \eqref{5.4} has non-zero determinant. Then \eqref{5.7} and \eqref{5.5} define an irrational vector $\beta$ 
and an invertible linear map $T$ such that
\begin{equation*}
T(\mathbb Z\alpha+\mathbb Z^d)\subset\mathbb Z\beta+\mathbb Z^d.
\end{equation*}
Conversely, any such $\beta$ and $T$ can be obtained in this way.
 \end{theorem}

\begin{remark}
In the next section we will need expressions for the determinant of the
maps $U$ and $T$ in terms of the parameters $A,p,q$ and $r$, so we
mention them now.
By the formula for the determinant of a block matrix, we have
\begin{equation}\label{eq:detu}
\det U = r\det A -\langle q,\adj(A)p\rangle
\end{equation}
where $\adj(A)$ is the adjugate matrix of $A$. We also have 
\begin{equation}\label{eq:dett}
\det T=\det A+\langle q,\adj(A)\beta\rangle=\frac{\det U}{r-\langle q,\alpha\rangle}.
\end{equation}
The first equality follows from \eqref{5.5}, for instance by Sylvester's determinant identity.
The second equality is obtained by using the expression \eqref{5.7} for $\beta$, together
with the fact that $\adj(A)\cdot A=(\det A)\cdot I$ (where $I$ is the identity matrix) and
equality \eqref{eq:detu}.
\end{remark}

\subsection{Linear maps and non-Riemann measurable sets}
The proof of \thmref{thm:main3} given above depends on the characterization of the 
Riemann measurable bounded remainder sets (\corref{cor:equidecomp}).
We do not know whether the result can be extended to all
bounded remainder sets, namely also those which are not Riemann measurable.
Nevertheless, we shall now present a different argument which shows that this is true
under the stronger assumption that
$T(\bz \va + \bz^d) = \bz \vb + \bz^d$.

 \begin{theorem}\label{thm7.2}
 Let $\alpha$ and $\beta$ be two irrational vectors in $\mathbb R^d,$ and let $T$ be an invertible linear map on $\mathbb R^d.$
Then the condition
\begin{equation}\label{eq:tabeqcond}
T(\mathbb Z\alpha+\mathbb Z^d)=\mathbb Z\beta+\mathbb Z^d
\end{equation}
is equivalent to the following one: for every bounded, measurable set $S$ in $\mathbb R^d$ (not necessarily Riemann measurable),
$S$ is a bounded remainder set with respect to $\alpha$ if and only if its image $T(S)$ is a bounded remainder set with respect to $\beta$.
 \end{theorem}

\begin{remark}
Condition \eqref{eq:tabeqcond} means that the mapping $U$ of \thmref{thm7.1}
is a \emph{bijection} over $\mathbb Z^d\times \mathbb Z$, which is the case if and only if
$\det U=\pm1$. Hence the irrational vectors $\beta$ and the invertible linear maps $T$ 
satisfying condition \eqref{eq:tabeqcond} for a given irrational vector $\va$
are parametrized by the $(d+1)\times (d+1)$ integer matrices with determinant $\pm 1$.
\end{remark}

\begin{remark}
Notice that also in \corref{cor:brsab} it was not required that the sets are Riemann measurable,
and that this case is not covered by \thmref{thm7.2}.
\end{remark}

\begin{proof}[Proof of \thmref{thm7.2}]
It will be enough if we prove the following claim: if condition \eqref{eq:tabeqcond} is satisfied,
and if $S$ is a BRS with respect to $\alpha$, then $T(S)$ is a BRS with respect to $\beta.$
Indeed, combining this with \thmref{thm:main3}, and applying the same considerations also
for the inverse map $T^{-1}$ in place of $T$, yields the full assertion of \thmref{thm7.2}.

We thus suppose that \eqref{eq:tabeqcond} holds, and let $S$ be a BRS with respect to $\alpha$. Denote
$$\nu_M(\beta,T(S),y)=\sum_{n'=0}^{M-1}\chi_{T(S)}(y+n'\beta)$$
where $y \in \br^d$, and observe that $\nu_M(\beta,T(S),y)$ is equal to the number of vectors $(m',n')\in\mathbb Z^d\times\mathbb Z$ satisfying
the two constraints
\begin{equation}\label{9.1}
0\le n'\le M-1,
\end{equation}
\begin{equation}\label{9.2}
y+n'\beta+m'\in T(S).
\end{equation}
By condition \eqref{eq:tabeqcond}, the mapping $U:(m,n)\mapsto(m',n')$ 
 is a bijection over $\mathbb Z^d\times \mathbb Z$. 
By \eqref{5.2} and \eqref{5.4}, the two constraints \eqref{9.1} and \eqref{9.2} may thus
be reformulated as
\begin{equation}\label{6.1}
0\le rn+\langle q,m\rangle\le M-1,
\end{equation}
\begin{equation}\label{6.2}
x+n\alpha+m\in S,
\end{equation}
where $x\in\mathbb R^d$ is the point such that $Tx=y.$

According to \eqref{eq:dett} we have
$(\det T)^{-1} (\det U) = r-\langle q,\alpha\rangle,$
and $\det U=\pm1$ since $U$ is a bijection over $\mathbb Z^d\times\mathbb Z$.
Hence another reformulation of the constraint \eqref{6.1} is
\begin{equation*}
0\le \pm\frac{n}{|\det T|}+\langle q,n\alpha+m\rangle\le M-1
\end{equation*}
where the $\pm$ is the sign of $(\det U)\cdot(\det T).$ In what follows, we shall
consider the case when this sign is positive (the other case can be treated similarly).

Assume that $y$ belongs to the unit cube $Q=[0,1)^d.$ Hence $x$ belongs to the bounded set $T^{-1}(Q).$ Since also $S$ is bounded, it follows from \eqref{6.2} that 
$$|\langle q,n\alpha+m\rangle|\le C_1,$$
where $C_1$ is a constant not depending on $M$ or $y$.
Thus, if we consider the following new set of constraints on the vector $(m,n)\in\mathbb Z^d\times \mathbb Z$, namely
\begin{equation}\label{6.4}
0\le\frac{n}{|\det T|}\le M-1,
\end{equation}
\begin{equation}\label{6.4a}
x+n\alpha+m\in S,
\end{equation}
then the number of solutions differs from $\nu_M(\beta,T(S),y)$ by at most some
constant $C_2$. But observe that the number of solutions to \eqref{6.4} and \eqref{6.4a} is just
$$\nu_N(\alpha,S,x)=\sum_{n=0}^{N-1}\chi_S(x+n\alpha),$$
where
\begin{equation}\label{6.5}
N= \big\lfloor|\det T|\cdot(M-1) \big\rfloor+1.
\end{equation}
To summarize, we have proved that
\begin{equation}\label{6.6}
|\nu_M(\beta,T(S),y)-\nu_N(\alpha,S,x)|\le C_2,
\end{equation}
where $N$ is related to $M$ by \eqref{6.5} and where $x$ is related to $y$ by $Tx=y.$

Now we can show that $T(S)$ is a BRS with respect to $\beta.$ Indeed, we have
\begin{align*}
|\nu_M(\beta,T(S),y)-M \mes T(S)| &  \le |\nu_M(\beta,T(S),y)-\nu_N(\alpha,S,x)| \\
  &+|\nu_N(\alpha,S,x)-N \mes S|\\
&+|N  \mes S-M  \mes T(S)|.
\end{align*}
The first summand on the right hand side is bounded by \eqref{6.6}.
The second summand is bounded by some constant $C_3$ for a.e.\ $x,$ since $S$ is a BRS with respect to $\alpha.$
The last summand is equal to 
$$(\mes S)\cdot\Big|N-|\det T|\cdot M\Big|\le C_4,$$
due to \eqref{6.5}. We conclude that
$$|\nu_M(\beta,T(S),y)-M \mes T(S)|\le C$$
for every $M$ and a.e.\ $y\in Q=[0,1)^d$, where the constant $C$ depends neither on $M$ nor on $y$.
 This shows that $T(S)$ is a BRS with respect to $\beta$, and concludes the proof.
\end{proof}


\section{Remarks}
\label{sec:remarks}

\subsection{}
An interesting question which is left open concerns the completeness of Hadwiger invariants with respect to the group of translations by vectors belonging to a general subgroup $\Gamma$ of $\mathbb R^d.$ The condition that $H_\Phi(S,\Gamma)=H_\Phi(S',\Gamma)$ for all $k$-flags $\Phi$ $(0\le k\le d-1)$ is necessary for two polytopes $S$ and $S'$ of the same volume to be equidecomposable using translations by vectors in $\Gamma.$ Is this condition also sufficient?

For our purpose, the case $\Gamma=\mathbb Z\alpha+\mathbb Z^d$ is important. An affirmative answer in this case would imply that the bounded remainder polytopes can be characterized by the condition that $H_\Phi(S, \, \mathbb Z\alpha+\mathbb Z^d)=0$ for any $k$-flag $\Phi$ $(0\le k\le d-1)$.

\subsection{}
Another problem concerns the characterization of bounded remainder sets which are not necessarily Riemann measurable. 
Recall that in Proposition \ref{prop:shiftsubset}, the sets $S$ and $S'$ were not required to be Riemann measurable.
Can one extend Theorem \ref{thm:main2} and prove that any two bounded remainder sets of the same measure
(not necessarily Riemann measurable) are equidecomposable
using translations by vectors in $\bz \va + \bz^d$ only?

An affirmative answer would allow, in particular, to extend Theorem \ref{thm:main3} to all
bounded remainder sets, including those which are not Riemann measurable. We proved this
in the special case when $T(\mathbb Z\alpha+\mathbb Z^d)=\mathbb Z\beta+\mathbb Z^d$
(\thmref{thm7.2}). It is also true if $T$ is the identity map (Corollary \ref{cor:brsab}).

\subsection{}
Let $A$ and $B$ be two bounded, measurable sets in $\mathbb R^d$ of the same measure.
One may ask when the difference in visiting times remains bounded as $n\to\infty$, that is
\begin{equation}\label{8.1}
\Big|
\sum_{k=0}^{n-1}\chi_A(x+k\alpha)-\sum_{k=0}^{n-1}\chi_B(x+k\alpha)
\Big| \leq C
 \quad (n=1,2,3,\dots) 
 \quad \text{a.e.\ $x \in \bt^d$.} 
\end{equation}
In particular, this condition is satisfied if $A$ and $B$ are two bounded remainder sets of the same measure.

The case when $A,B$ are two intervals on $\mathbb R$ was considered
by Furstenberg, Keynes and Shapiro \cite{furstenberg} who
characterized the pairs of intervals with this property:

\emph{Two intervals $A=[a,a+h)$ and $B=[b,b+h)$ satisfy condition \eqref{8.1} if and only if 
$h\in\mathbb Z\alpha+\mathbb Z$ or $b-a \in \mathbb Z\alpha+\mathbb Z.$}

One can see that in our proof of Theorem \ref{thm:main2}, the assumption that $A,B$ are bounded remainder sets of the same measure was merely used as a sufficient condition for \eqref{8.1}. Hence Theorem \ref{thm:main2} remains true under this weaker assumption. Moreover, the converse statement is also true in this case, and can be proved in a similar way to Proposition \ref{prop:shiftsubset}. In other words, we have the following more general version of Theorem \ref{thm:main2}.

\begin{theorem}\label{thm8.1}
Let $A$ and $B$ be two Riemann measurable sets in $\mathbb R^d$.
Then condition \eqref{8.1} is satisfied if and only if $A,B$ are equidecomposable (by Riemann measurable pieces) using translations by vectors in $\mathbb Z\alpha+\mathbb Z^d$.
\end{theorem}

If $A,B$ are two polytopes in $\mathbb R^d$ satisfying \eqref{8.1}, then
they are equidecomposable by pieces which are also polytopes. 
It follows that:

\begin{theorem}\label{thm8.2}
For two polytopes $A,B$ in $\mathbb R^d$ to satisfy \eqref{8.1} it is necessary that
$$H_\Phi(A, \, \mathbb Z\alpha+\mathbb Z^d)=H_\Phi(B, \, \mathbb Z\alpha+\mathbb Z^d)$$
for any $k$-flag $\Phi$ $(0\le k\le d-1).$
\end{theorem}

In the case when $A,B$ are two intervals on $\mathbb R$ this yields the necessity part in the 
characterization obtained in \cite{furstenberg} (the sufficiency part is easy to prove).

\subsection{}
A bounded, measurable function $f$ on $\mathbb T^d$ is called a \emph{bounded remainder function} if there is a constant 
$C=C(f,\alpha)$ such that 
$$\Big|\sum_{k=0}^{n-1}f(x+k\alpha)-n\int f\Big|\le C 
 \quad (n=1,2,3,\dots) 
 \quad \text{a.e.\ $x \in \bt^d$.} 
$$
Thus a set $S$ is a BRS if and only if $f=\chi_S$ is a bounded remainder function.

Bounded remainder functions have been studied by various authors. For example,
Oren \cite[Theorem A]{oren} characterized the piecewise constant functions of bounded remainder in dimension one. 
Other results may be found in \cite{schoissen} and the references therein.

The case when $f$ is a $\mathbb Z$-valued function
is basically covered by the theory of bounded remainder sets. Indeed, for such $f$ there exists a bounded remainder set $S$ such that
$$f(x)-\int f=\chi_S(x)-\mes S,$$
and if $f$ is Riemann integrable then $S$ may be chosen to be a Riemann measurable set.

On the other hand, some of the basic results on bounded remainder sets extend,
with essentially the same proofs, to general functions of bounded remainder.
In particular, this is so for the equivalence of the bounded remainder property and
the existence of a bounded transfer function. Namely, $f$ is a bounded remainder function
if and only if there exists a bounded, measurable function $g$ on $\bt^d$ satisfying the cohomological equation
\begin{equation*}
f(x) - \int f = g(x) - g(x-\va) \quad \text{a.e.}
\end{equation*}

By a classical theorem of Gottschalk and Hedlund \cite[Theorem 14.11]{gottschalk}
if a bounded remainder function $f$ is continuous, then it admits a continuous transfer function $g$.
What about the analog of Theorem \ref{thm:Rtransfer}\,?
That is, if $f$ is a Riemann integrable function of bounded remainder, does it have a Riemann integrable transfer function?
See \cite[Theorem 1]{schoissen} where such a result is proved for the class of piecewise continuous functions   in dimension one.


\begin{thebibliography}{10}

\bibitem{bolle}
U.~{Bolle}, \emph{{On multiple tiles in $E\sp 2$}}, {Intuitive geometry.
  Proceedings of the 3rd international conference held in Szeged, Hungary, from
  2 to 7 September, 1991}, Amsterdam: North-Holland; Budapest: J\'anos Bolyai
  Mathematical Society, 1994, pp.~39--43.
  
\bibitem{boltianski}
V.~Boltianski, \emph{Hilbert's third problem}, Wiley, 1978.

\bibitem{cassels}
J.~W.~S. Cassels, \emph{{An Introduction to the Geometry of Numbers}},
  Springer, 1997.

\bibitem{concini}
C.~De Concini and C.~Procesi, \emph{{Topics in hyperplane arrangements,
  polytopes and box-splines}}, {New York, NY: Springer}, 2011.

\bibitem{erdos}
P.~{Erd\H{o}s}, \emph{{Problems and results on Diophantine approximations}},
  {Compos. Math.} \textbf{16} (1964), 52--65.

\bibitem{ferenczi}
S.~{Ferenczi}, \emph{{Bounded remainder sets}}, {Acta Arith.} \textbf{61}
  (1992), no.~4, 319--326.

\bibitem{furstenberg}
H.~Furstenberg, H.~Keynes, and L.~Shapiro, \emph{{Prime flows in topological
  dynamics}}, Isr. J. Math. \textbf{14} (1973), 26--38.

\bibitem{gottschalk}
W.H. {Gottschalk} and G.A. {Hedlund}, \emph{{Topological dynamics}},
  {(Colloquium Publications of the American Mathematical Society (AMS). Vol.
  36.) Providence, R.I.: American Mathematical Society (AMS). VIII, 151 p.},
  1955.

\bibitem{grhar}
P. Griffiths and J. Harris, 
\emph{Principles of algebraic geometry}, Wiley, New York, 1978.

\bibitem{hadwiger2}
H.~{Hadwiger}, \emph{{Translationsinvariante, additive und schwachstetige Polyederfunktionale}}, 
{Arch. Math.} \textbf{3} (1952), 387--394 (German).

\bibitem{hadwigerbook}
H.~{Hadwiger}, \emph{{Vorlesungen \"uber Inhalt, Oberfl\"ache und Isoperimetrie}},
  {Die Grundlehren der Mathematischen Wissenschaften. 93. Berlin-
  G\"ottingen-Heidelberg: Springer-Verlag XIII, 312 S.}, 1957.

\bibitem{hadwiger}
H.~{Hadwiger}, \emph{{Translative Zerlegungsgleichheit der Polyeder des
  gew\"ohnlichen Raumes}}, J. Reine Angew. Math. \textbf{233} (1968), 200--212
  (German).

\bibitem{glur}
H.~Hadwiger and P.~Glur, \emph{{Zerlegungsgleichheit ebener Polygone}}, Elem.
  Math. \textbf{6} (1951), 97--106 (German).


\bibitem{halasz}
G.~Hal\'{a}sz, \emph{{Remarks on the remainder in Birkhoff's ergodic
  theorem}}, Acta Math. Acad. Sci. Hung. \textbf{28} (1976), 389--395.

\bibitem{hartman}
S.~Hartman, Colloq. Math. \textbf{1} (1947), 239--240 (French).

\bibitem{haynes}
A.~Haynes and H.~Koivusalo, \emph{Constructing bounded remainder sets and
  cut-and-project sets which are bounded distance to lattices}, (2014), \texttt{arXiv:1402.2125}.

\bibitem{hecke}
E.~{Hecke}, \emph{{\"Uber analytische Funktionen und die Verteilung von Zahlen
  mod. eins}}, {Abh. Math. Semin. Univ. Hamb.} \textbf{1} (1921), 54--76
  (German).

\bibitem{jessen2}
B.~Jessen, \emph{Zur Algebra der Polytope}, Nachr. Akad. Wiss. Göttingen Math.-Phys. Kl. II (1972), 47--53  (German).

\bibitem{jessen}
B.~Jessen and A.~Thorup, \emph{{The algebra of polytopes in affine spaces}},
  Math. Scand. \textbf{43} (1978), 211--240.

\bibitem{kesten}
H.~{Kesten}, \emph{{On a conjecture of Erd\H{o}s and Sz\"usz related to uniform
  distribution mod 1}}, {Acta Arith.} \textbf{12} (1966), 193--212.

\bibitem{liardet}
P.~{Liardet}, \emph{{Regularities of distribution}}, {Compos. Math.}
  \textbf{61} (1987), 267--293.
  
\bibitem{murner}
P.~{M\"urner}, \emph{{Translative Zerlegungsgleichheit von Polytopen}}, {Arch.
  Math.} \textbf{29} (1977), 218--224 (German).  

\bibitem{oren}
I.~Oren, \emph{Admissible functions with multiple discontinuities}, Isr. J.
  Math. \textbf{42} (1982), 353--360.

\bibitem{ostrowski2}
A.~Ostrowski, \emph{{Mathematische Miszellen IX: Notiz zur Theorie der
  Diophantischen Approximationen}}, Jahresber. Dtsch. Math.-Ver. \textbf{36}
  (1927), 178--180 (German).

\bibitem{ostrowski}
\bysame, \emph{{Mathematische Miszellen. XVI: Zur Theorie der linearen
  diophantischen Approximationen}}, Jahresber. Dtsch. Math.-Ver. \textbf{39}
  (1930), 34--46 (German).

\bibitem{petersen}
K.~Petersen, \emph{{On a series of cosecants related to a problem in ergodic
  theory}}, Compos. Math. \textbf{26} (1973), 313--317.

\bibitem{rauzy1}
G.~{Rauzy}, \emph{{Nombres alg\'ebriques et substitutions}}, {Bull. Soc. Math.
  Fr.} \textbf{110} (1982), 147--178 (French).

\bibitem{rauzy3}
G.~{Rauzy},  \emph{{Ensembles \`a restes born\'es}}, {S\'emin. Th\'eor.
  Nombres, Univ. Bordeaux I 1983-1984, Exp. No.24, 12 p.}, 1984 (French).

\bibitem{sah}
C.-H. {Sah}, \emph{{Hilbert's third problem: scissors congruence}}, {Research
  Notes in Mathematics. 33. San Francisco, London, Melbourne: Pitman Advanced
  Publishing Program. X, 188 p.}, 1979.

\bibitem{schoissen}
J.~Schoissengeier, \emph{{Regularity of distribution of
  $(n\alpha)$-sequences}}, Acta Arith. \textbf{133} (2008), no.~2, 127--157.

\bibitem{shephard}
G.~C. Shephard, \emph{{Combinatorial properties of associated zonotopes}},
  Can. J. Math. \textbf{26} (1974), 302--321.

\bibitem{shutov}
A.~V. {Shutov}, \emph{{On a family of two-dimensional bounded remainder
  sets}}, {Chebyshevski\u{\i} Sb.} \textbf{12} (2011), no.~4(40), 264--271
  (Russian).

\bibitem{sydler}
J.-P.  Sydler, \emph{Conditions n\'ecessaires et suffisantes pour l'\'equivalence des poly\`edres de l'espace euclidien \`a trois dimensions},
Comment. Math. Helv. \textbf{40} (1965), 43--80 (French) .

\bibitem{szusz}
P.~Sz{\"{u}}sz, \emph{{\"{U}}ber die {V}erteilung der {V}ielfachen einer
  komplexen {Z}ahl nach dem {M}odul des {E}inheitsquadrats}, Acta Math. Acad.
  Sci. Hungar. \textbf{5} (1954), 35--39 (German).

\bibitem{szusz2}
P.~Sz{\"{u}}sz, \emph{{L{\"{o}}sung eines Problems von Herrn Hartman}}, Stud. Math.
  \textbf{15} (1955), 43--55 (German).

\bibitem{zhuravlev}
V.~G. Zhuravlev, \emph{Bounded remainder polyhedra}, {Proc. Steklov Inst.
  Math.} \textbf{280} (2013), 71--90.

\bibitem{ziegler}
G.~M. {Ziegler}, \emph{{Lectures on polytopes}}, Berlin: Springer-Verlag,
  1995.
  
\end{thebibliography}
\end{document}